\documentclass[a4paper,10pt,reqno]{amsart}

\usepackage[english]{babel}
\usepackage[utf8]{inputenc}
\usepackage{amsmath,amssymb,amsfonts,amsthm,amscd}
\usepackage{color}
\usepackage{graphicx}
\usepackage{enumitem}
\usepackage{hyperref}
\setlist[itemize]{topsep=0.2em, itemsep=0.2em, leftmargin=2em}
\setlist[enumerate]{topsep=0.2em, itemsep=0.2em, leftmargin=2em}
\setlist[description]{topsep=0.2em, itemsep=0.2em, leftmargin=2em}

\newcommand{\df}{\mathrm{d}}

\DeclareMathOperator{\Jac}{Jac}

\DeclareMathOperator{\area}{area}

\DeclareMathOperator{\diver}{div}

\newtheorem{theorem}{Theorem}[section]
\newtheorem{proposition}[theorem]{Proposition}

\newtheorem{lemma}[theorem]{Lemma}

\theoremstyle{definition}
  
	\newtheorem{example}[theorem]{Example}

\theoremstyle{remark}
  \newtheorem{remark}[theorem]{Remark}

\numberwithin{equation}{section}

\numberwithin{equation}{section}
\title[A duality for prescribed mean curvature graphs]{A duality for prescribed mean curvature graphs in Riemannian and Lorentzian Killing submersions}
\date{}

\author{Andrea Del Prete}
\address{Andrea Del Prete, Dipartimento di Ingegneria e Scienze dell'Informazione e Matematica\\
  Universit\`{a} dell'Aquila, Italy}
\email{andrea.delprete@graduate.univaq.it}

\author{Hojoo Lee}
\address{Hojoo Lee, 
Department of Mathematics and Institute of Pure and Applied Mathematics, Jeonbuk National University, South Korea}
\email{compactkoala@gmail.com, kiarostami@jbnu.ac.kr}
 
\author{Jos\'{e} M. Manzano}
\address{Jos\'{e} M. Manzano, Departamento de Matemáticas, Universidad de Jaén, Spain.}
\email{jmprego@ujaen.es}



\begin{document}

\begin{abstract}
We develop a conformal duality for spacelike graphs in Riemannian and Lorentzian three-manifolds that admit a Riemannian submersion over a Riemannian surface whose fibers are the integral curves of a Killing vector field, which is timelike in the Lorentzian case. The duality swaps mean curvature and bundle curvature and sends the length of the Killing vector field to its reciprocal while keeping invariant the base surface. We obtain two consequences of this result. On the one hand, we find entire graphs in Lorentz--Minkowski space $\mathbb{L}^3$ with prescribed mean curvature a bounded function $H\in\mathcal{C}^\infty(\mathbb{R}^2)$ with bounded gradient. On the other hand, we obtain conditions for existence and non existence of entire graphs which are related to a notion of critical mean curvature.
\end{abstract}

\maketitle

\section{Introduction}

The mean curvature is one of the most widely studied functionals in Surface Theory in different ambient $3$-manifolds. This quasilinear elliptic operator acquires a divergence form when the surface is a graph in Euclidean space $\mathbb{R}^3$ acting on the function that defines the graph. The literature on questions related to the mean curvature of graphs is too vast to be cited here. However, the starting point of this work is the fact that a minimal graph in Euclidean space has divergence zero and can be transformed it into a maximal (spacelike) graph in Lorentz--Minkowski space $\mathbb{L}^3$ by means of the Poincaré lemma. This clever trick is usually attributed to Calabi~\cite{Calabi}, but the work of Catalan~\cite{Catalan} evidences that the very same argument dates back to the nineteenth century.

The second author~\cite{Lee11a} extended the Calabi duality to the case of homogeneous spaces $\mathbb{E}(\kappa,\tau)$ with isometry group of dimension $4$ (and their Lorentzian counterparts $\mathbb{L}(\kappa,\tau)$). These spaces fiber over $\mathbb{M}^2(\kappa)$, the complete surface of constant curvature $\kappa$, with constant bundle curvature $\tau$ and totally geodesic fibers tangent to a unitary Killing vector field. In this context, the mean curvature equation can be also transformed into a divergence zero equation, and~\cite{Lee11a} uses again Poincaré lemma to get a duality between graphs with constant mean curvature $H$ in $\mathbb{E}(\kappa,\tau)$ and spacelike graphs with constant mean curvature $\tau$ in $\mathbb{L}(\kappa,H)$. The case of minimal surfaces in $\mathbb{S}^2\times\mathbb{R}$ (i.e., the particular case $\kappa=1$ and $\tau=H=0$) was actually discovered earlier by Albujer and Al\'{i}as~\cite{AA}. In~\cite{LeeMan}, the last two authors realized that the same technique can be generalized to $3$-manifolds with unitary Killing vector fields by prescribing non-necessarily constant mean and bundle curvature functions that are swapped by the duality. In the present work, we will move forward to obtain a duality under the presence of any Killing vector field with no zeros, not necessarily of constant length. We believe this is possibly the most general scenario where the mean curvature still acquires a divergence type equation, plus there is a notion of bundle curvature that also admits a divergence type expression.

We will actually consider Riemannian and Lorentzian Killing submersions, which locally model any $3$-manifold with a Killing vector field with no zeros. Such a submersion $\pi:\mathbb{E}\to M$, being $\mathbb{E}$ a simply connected Riemannian or Lorentzian $3$-manifold, is determined by the simply connected Riemannian base surface $M$ and two geometric functions: the bundle curvature $\tau$ and the Killing length $\mu$, both of which can be prescribed arbitrarily (being the latter positive). This was proved by Lerma and the third author~\cite{LerMan} (see Theorem~\ref{thm:Killing-classification}). The total spaces of these submersions will be therefore denoted by $\mathbb{E}(M,\tau,\mu)$ and $\mathbb{L}(M,\tau,\mu)$ depending on whether they are Riemannian or Lorentzian. Note that $\mathbb{E}(\kappa,\tau)=\mathbb{E}(\mathbb{M}^2(\kappa),\tau,1)$ and $\mathbb{L}(\kappa,\tau)=\mathbb{L}(\mathbb{M}^2(\kappa),\tau,1)$, but also these spaces contain the warped products $M\times_\mu\mathbb{R}$ (which in turn contain the Riemannian or Lorentzian products $M\times\mathbb{R}$ whenever $\mu\equiv 1$) and all homogeneous three manifolds, which admit many Killing submersion structures depending on the choice of the Killing vector field. 

There is a natural notion of graph in $\mathbb{E}(M,\tau,\mu)$ and $\mathbb{L}(M,\tau,\mu)$ defined as a section of the submersion over $M$. Our main result (Theorem~\ref{thm:duality}) is a conformal duality between entire graphs in $\mathbb{E}(M,\tau,\mu)$ with mean curvature $H$ and entire spacelike graphs in $\mathbb{L}(M,H,\mu^{-1})$ with mean curvature $\tau$. This is a very general result that covers all previously known cases since $M$ is an arbitrary simply connected surface and $H,\tau,\mu\in\mathcal{C}^\infty(M)$ are also arbitrary (such that $\mu>0$). 

This yields a geometric connection between two apparently different theories which helps understand some geometric features. For instance, Fernández and Mira's classification~\cite{FM} of entire minimal graphs in Heisenberg space $\mathrm{Nil}_3=\mathbb{E}(0,\frac{1}{2})$ becomes transparent by considering the dual entire spacelike graphs in $\mathbb{L}^3$ with constant mean curvature $\frac{1}{2}$, see~\cite{Man19}. Also, the third author and Nelli~\cite{MN} showed that gradient estimates for entire minimal graphs in $\mathrm{Nil}_3$ are related to the Cheng and Yau's estimates~\cite{CY2} for the dual graphs in $\mathbb{L}^3$. In~\cite{LerMan}, the duality was used to show the existence of entire minimal graphs in Riemannian Killing submersions over compact surfaces using the existence results for prescribed mean curvature graphs in Lorentzian warped products obtained by Gerhardt~\cite{Ger}. In~\cite{LeeMan}, the duality revealed that many Lorentzian Killing submersions do not admit any complete spacelike surface by an extension of a classical argument of Heinz~\cite{He55} for constant mean curvature graphs in the Riemannian setting (see also Theorem~\ref{thm:NON}).

In this paper, as a first application of the duality, we will obtain entire spacelike graphs in Lorentz--Minkowski space $\mathbb{L}^3=\mathbb{L}(\mathbb{R}^2,0,1)$ with bounded prescribed mean curvature $H\in\mathcal{C}^\infty(\mathbb{R}^2)$ such that $\nabla H$ is also bounded, see Theorem~\ref{thm:L3}. This is achieved by constructing the dual entire minimal graphs in $\mathbb{E}(\mathbb{R}^2,H,1)$ using the theory of divergence lines, developed by Mazet, Rosenberg and Rodríguez~\cite{MRR} and adapted to the case of Killing submersions by Nelli and the first and third authors~\cite{DMN}. In our proof, we have extended some of the results in~\cite{DMN} to take limits in $3$-manifolds whose geometry is not necessarily bounded by a diagonal argument with respect to an exhaustion by relatively compact domains. In $\mathbb{E}(\mathbb{R}^2,H,1)$, we discard the possible divergence lines by applying Mazet's halfspace theorem~\cite{Mazet}, and it is precisely at this point where we use that $H$ and $\nabla H$ are bounded. 

In particular, we give a partial answer to a conjecture stated in~\cite{LeeMan} that there are entire graphs in $\mathbb{L}^3$ with any prescribed mean curvature $H\in\mathcal{C}^\infty(\mathbb{R}^2)$. We also prove this conjecture in Lorentzian warped products $\mathbb{E}(M,0,\mu)$ in which $M$, $\mu$ and $H$ are all invariant by rotations or translations with no assumptions on their growth, see Proposition~\ref{prop:rotations-translations}. This means that our hypothesis in Theorem~\ref{thm:L3} are not sharp. In higher dimensions, this problem has been discussed in the literature as related to the Born--Infeld equation in which the mean curvature plays the role of the density of charge of an electrostatic physical system, and a solution is usually required to vanish at infinity (e.g., see~\cite{BDP,BIMM} and the references therein). In our approach, we are able to prescribe the normal at any point of the base by means of a degree argument, see Lemma~\ref{lem:angle-1} and Remark~\ref{rmk:prescribed-normal}.

Our second application of the duality is about the non-existence of entire graphs. In theorem~\ref{thm:NON}, we prove that $\mathbb{E}(M,\tau,\mu)$ does not admit any entire graph with $\inf_M|H|>\frac{1}{2}\mathrm{Ch}(M,\mu)$ and the dual statement that $\mathbb{L}(M,\tau,\mu^{-1})$ does not admit complete spacelike surfaces (of any mean curvature) if $\inf_M|\tau|>\frac{1}{2}\mathrm{Ch}(M,\mu)$. Here, $\mathrm{Ch}(M,\mu)$ is a constant that we have named \emph{Cheeger constant with density} $\mu$, see Equation~\eqref{eqn:cheeger}. Theorem~\ref{thm:NON} had already been proved in~\cite{LeeMan} in the unitary case $\mu\equiv 1$, in which $\mathrm{Ch}(M,\mu)$ is the classical Cheeger constant. In the case of the homogeneous $\mathbb{E}(\kappa,\tau)$-spaces, the value $H_0=\frac{1}{2}\mathrm{Ch}(M,\mu)$ is the so-called \emph{critical mean curvature}. If $H\leq H_0$, then there are entire graphs with constant mean curvature $H$ in $\mathbb{E}(\kappa,\tau)$; on the contrary, if $H>H_0$, then there are compact bigraphs with constant mean curvature $H$. This dichotomy plays a crucial role in the solution of the Hopf problem in homogeneous $3$-manifolds, see~\cite{MMPR}. Motivated by this fact, we have investigated if $H_0=\frac{1}{2}\mathrm{Ch}(M,\mu)$ distinguishes the existence of entire graphs and compact surfaces in $\mathbb{E}(M,\tau,\mu)$. In Theorem~\ref{thm:rotational}, we solve completely this problem in any rotationally invariant Riemannian warped product $\mathbb{E}(M,0,\mu)$. Remarkably, we find that some specific values of $H>H_0$ give rise to rotationally invariant nonentire complete graphs, which we call $H$-\emph{cigars}, see Figure~\ref{fig:cigar}. We believe that the constant $\frac{1}{2}\mathrm{Ch}(M,\mu)$ is related to the critical mean curvature in all homogeneous $3$-manifolds for any of their (many) Killing submersion structures.

We remark that all the rotational examples we have obtained in $\mathbb{E}(M,0,\mu)$ or $\mathbb{L}(M,0,\mu)$ in the proofs of Proposition~\ref{prop:rotations-translations} and Theorem~\ref{thm:rotational} have been obtained by means of the duality. It is hard to get a direct solution of the associated \textsc{ode} since we recall that $M$, $\mu$ and $H$ are arbitrary (rotationally symmetric) objects. It is also important to mention that Theorem~\ref{thm:L3} uses strongly the duality since we it transform the prescribed mean curvature problem in the Lorentzian setting into a problem for minimal graphs in the Riemannian setting, where there are many more results that come in handy to analyze convergence.

The paper is organized as follows. In Section~\ref{sec:preliminaries}, we collect some results about Riemannian and Lorentzian Killing submersions. To compute the mean curvature of Killing graphs, we use a coordinate free approach based on the calculus of variations, but also consider graphs in coordinates, which is quite helpful when applying the duality. We would like to highlight that the curvature tensor of a Killing submersion is also computed (see Proposition~\ref{prop:curvature}). Section~\ref{sec:duality} is devoted to prove the duality. In Section~\ref{sec:L3}, we give the first application to the existence of entire graphs in $\mathbb{L}^3$; in Section~\ref{sec:non}, we give the existence and non-existence results concerning the Cheeger constant and the critical mean curvature.

\medskip
\noindent\textbf{Acknowledgement.} 
The first author is supported by a INdAM--GNSAGA grant (codice CUP\_E55F22000270001). The third author is supported by the Ramón y Cajal fellowship RYC2019-027658-I and the project PID2019.111531GA.I00, both funded by MCIN/AEI/10.13039/501100011033, as well as by the FEDER-UJA project No.\ 1380860.

\section{Preliminaries on Killing graphs}\label{sec:preliminaries}

Let $\pi:\mathbb{E}\to M$ be a differentiable submersion from a Riemannian or Lorentzian $3$-manifold $\mathbb{E}$ onto a  Riemannian surface $M$, both of them connected and orientable. A tangent vector $v$ will be called \textit{vertical} when $v\in\ker(\df\pi)$ and \textit{horizontal} when $v\in\ker(\df\pi)^\bot$. The submersion $\pi$ is called a Riemannian (resp.\ Lorentzian) Killing submersion if $\mathbb{E}$ is Riemannian (resp.\ Lorentzian), $\pi$ preserves the length of horizontal vectors, and the fibers of $\pi$ are the integral curves of a complete Killing vector field $\xi\in\mathfrak{X}(\mathbb{E})$ with no zeros. In the Lorentzian case, this implies that $\xi$ is a timelike vector field. The $1$-parameter group of isometries associated to $\xi$ will be denoted by $\{\phi_t\}_{t\in\mathbb{R}}$, and its elements will be called \textit{vertical translations}. Note that $\phi_t$ is defined for all values of $t\in\mathbb{R}$ since we have assumed $\xi$ is complete.

In this setting, we can consider the \emph{connection $1$-form} $\alpha$ in $\mathbb{E}$ defined as $\alpha(X)=\langle X,\xi\rangle$, whose differential is given by $\df\alpha(X,Y)=\langle\nabla_X\xi,Y\rangle$ for all vector fields $X$ and $Y$, and it is known as \emph{curvature $2$-form} in $\mathbb{E}$. The fact that $\xi$ is Killing implies that $\df\alpha$ is skew-symmetric, so the function $\tau\in\mathcal{C}^\infty(\mathbb{E})$ given by
\begin{equation}\label{eqn:tau}
\tau(q)=\frac{1}{\|\xi_q\|^2}\,\df\alpha(v,v\times\xi),\quad q\in\mathbb{E},
\end{equation}
depends neither on the choice of a horizontal vector $v\in T_q\mathbb{E}$ with $\|v\|=1$ nor on rescaling $\xi$ by a constant factor. Here, $\times$ stands for the cross product in $\mathbb{E}$ defined such that $\langle u\times v,w\rangle=\det_{\mathbb B}(u,v,w)$ whenever $u,v,w\in T_q\mathbb{E}$ are expressed in coordinates as column vectors in a positively oriented orthonormal basis $\mathbb B$.

The function $\tau$ is called the bundle curvature of the Killing submersion and is constant along the fibers of $\pi$ due to the fact that each $\phi_t$ is an isometry of $\mathbb{E}$ such that $(\phi_t)_*\xi=\xi$ and, in particular, $(\phi_t)_*\omega=\omega$. Moreover, we will also consider the Killing length $\mu=|\langle\xi,\xi\rangle|^{1/2}\in\mathcal{C}^\infty(\mathbb{E})$, which is a positive function also constant along fibers. It follows that $\tau$ and $\mu$ can be viewed as functions in the base surface $M$, which will be also denoted by $\tau,\mu\in\mathcal{C}^\infty(M)$.

\begin{theorem}[{\cite{LerMan}}]\label{thm:Killing-classification}
Let $M$ be a simply connected Riemannian surface and let $\tau,\mu\in\mathcal{C}^\infty(M)$ such that $\mu>0$. Then, there exists a Riemannian (resp.\ Lorentzian) Killing submersion $\pi:\mathbb{E}\to M$ such that
\begin{enumerate}[label=\emph{(\arabic*)}]
 \item $\mathbb{E}$ is simply connected,
 \item $\tau$ is the bundle curvature of $\pi$, and
 \item $\mu$ is the length of a Killing field $\xi$ whose integral curves are the fibers of $\pi$.
\end{enumerate}
Moreover, $\pi:\mathbb{E}\to M$ is unique in the sense that if $\pi':\mathbb{E}'\to M$ is another Riemannian (resp.\ Lorentzian) Killing submersion satisfying conditions (1), (2) and (3) above, then there exists an isometry $T:\mathbb{E}\to\mathbb{E}'$ such that $\pi'\circ T=\pi$.
\end{theorem}

Accordingly, given a simply connected surface $M$ and $\tau,\mu\in\mathcal{C}^\infty(M)$ such that $\mu>0$, the unique simply connected Riemannian (resp.\ Lorentzian) $3$-manifold which admits a Killing submersion over $M$ with bundle curvature $\tau$ and Killing length $\mu$ will be denoted by $\mathbb{E}(M,\tau,\mu)$ (resp. $\mathbb L(M,\tau,\mu)$).

\subsection{The mean curvature of a Killing graph}

A (Killing) graph in a Killing submersion $\pi:\mathbb{E}\to M$ is a smooth section over an open subset $U\subset M$. If we prescribe a smooth zero section $F_0:U\to\mathbb{E}$, then such a graph can be parametrized as $F_u:U\to\mathbb{E}$ with $F_u(p)=\phi_{u(p)}(F_0(p))$ for some $u\in\mathcal{C}^\infty(U)$, where $\{\phi_t\}$ is the group of vertical translations. In the sequel, we will assume that the fibers of $\pi$ have infinite length, which implies the existence of global smooth sections, see~\cite{LerMan}.

Given $u\in\mathcal{C}^\infty(U)$, we will denote the graph spanned by $F_u$ by $\Sigma_u$, which will be assumed \emph{spacelike}, i.e., the restriction of the metric of $\mathbb{E}$ is positive definite. Following the ideas in~\cite{LerMan}, we will consider the functions $\overline u\in\mathcal{C}^\infty(\mathbb{E})$ defined by $\overline u=u\circ\pi$ and $d\in\mathcal{C}^\infty(\mathbb{E})$ defined implicitly by $\phi_{d(q)}(F_0(\pi(q)))=q$, i.e., $d(q)$ is the signed distance along a fiber from the initial section to $q$. Therefore, the upward pointing unit normal to $\Sigma_u$ can be expressed as $N=\epsilon\overline\nabla(d-\overline u)/\|\overline\nabla(d-\overline u)\|_{\mathbb{E}}$, where $\overline\nabla$ and $\|\cdot\|_{\mathbb{E}}$ stand for the gradient and norm in $\mathbb{E}$, respectively. 

Note that $\langle\overline\nabla d,\xi\rangle=\xi(d)=1$ by definition of $d$ and $\langle\overline\nabla \overline u,\xi\rangle=0$ since $\overline u$ is constant along the fibers of $\pi$. Therefore, we can decompose in vertical and horizontal components $\overline\nabla(d-\overline u)=\frac{\epsilon}{\mu^2}\xi+(\overline\nabla(d-\overline u))^h$. It follows from the orthogonality of the vertical and horizontal components that 
\begin{equation}\label{eqn:causality}
\|\overline\nabla(d-\overline u)\|_{\mathbb{E}}^2=\tfrac{\epsilon}{\mu^2}+\|(\overline\nabla(d-\overline u))^h\|_{\mathbb{E}}^2=\tfrac{\epsilon}{\mu^2}+\|\nabla u-Z\|^2,
\end{equation}
where $Z=\pi_*(\overline\nabla d)$ is a vector field on $U\subset M$ not depending on $u$. Here, $\nabla$ and $\|\cdot\|$ denote the gradient and norm in $M$, respectively. We also define $Gu=\nabla u-Z$, usually known as the \emph{generalized gradient} of $u$, see~\cite{LerMan,DMN}. Observe that $\overline\nabla(d-\overline u)$ is timelike in the Lorentzian case ($\epsilon=-1$), which amounts to saying that the right-hand side in~\eqref{eqn:causality} is negative, i.e., the spacelike condition is equivalent to $1+\epsilon\mu^2\|Gu\|^2>0$. This also means that we have to add a factor $\epsilon$ before taking square roots to get rid of the square in the left-hand side of~\eqref{eqn:causality}. Consequently, the angle function $\nu=\langle N,\xi\rangle$ of $\Sigma_u$ can be computed as
\begin{equation}\label{eqn:angle-function}
\nu=\frac{\epsilon\langle\overline\nabla(d-\overline u),\xi\rangle}{\|\overline\nabla(d-\overline u)\|}=\frac{\epsilon\mu}{\sqrt{1+\epsilon\mu^2\|Gu\|^2}}.
\end{equation}
Note that $0<\nu\leq\mu$ if $\epsilon=1$, whereas $\nu\leq-\mu$ if $\epsilon=-1$.

Since $\Sigma_u$ is a section of $\pi$, the projection $\pi|_{\Sigma_u}:\Sigma_u\to U$ is a diffeomorphism and the area element of $\Sigma_u$ over $U$ can be computed as the Jacobian of $\pi|_{\Sigma_u}$. Let $\{v_1,v_2\}$ be an orthonormal basis of $T_q\Sigma_u$ at some $q\in\Sigma_u$ such that $v_1$ is horizontal, and let $h\in T_q\mathbb{E}$ be an horizontal unit vector such that $\{v_1,h\}$ is also orthonormal. Since $\xi$, $N$, $v_2$ and $h$ are coplanar (all of them are orthogonal to $v_1$), we can easily express $N=\epsilon\frac{\nu}{\mu^2}\xi\pm \frac{1}{\mu}\sqrt{\epsilon(\mu^2-\nu^2)}h$ and then work out the orthogonal vector $v_2=\frac{1}{\mu^2}\sqrt{\epsilon(\mu^2-\nu^2)}\xi\mp\frac{\nu}{\mu} h$, where the signs depend on the choice of $h$. Since $\pi$ is a Riemannian submersion, we deduce that $\{v_1,v_2\}$ projects to an orthogonal basis $\{\df\pi_q(v_1),\df\pi_q(v_2)\}$ such that $\|\df\pi_q(v_1)\|=1$ and $\|\df\pi_q(v_2)\|=\frac{|\nu|}{\mu}$. This implies that 
\begin{equation}\label{eqn:jacobian-projection}
|\Jac(\pi|_{\Sigma_u})|=\frac{|\nu|}{\mu}=\frac{1}{\sqrt{1+\epsilon\mu^2\|Gu\|^2}}.
\end{equation}
For each relatively compact subdomain $\Omega\subset\overline\Omega\subset U$, a direct change of variables using~\eqref{eqn:jacobian-projection} yields the desired area element:
\begin{equation}\label{eqn:area-element}\area(\Sigma_u\cap\pi^{-1}(\Omega))=\int_\Omega\sqrt{1+\epsilon\mu^2\|Gu\|^2}.
\end{equation}

\begin{proposition}\label{prop:mean-curvature}
The mean curvature of a Killing graph parametrized by a function $u\in\mathcal{C}^2(U)$ under the above assumptions is given by
\[H=\frac{1}{2\mu}\diver\left(\frac{\mu^2 Gu}{\sqrt{1+\epsilon\mu^2\|Gu\|^2}}\right),\]
where the divergence is computed in $M$.
\end{proposition}

\begin{proof}
Let $f\in\mathcal{C}^\infty_0(U)$ be a smooth function that vanishes outside a relatively compact open subset $\Omega\subset\overline\Omega\subset U$, and consider the functional $A_f(t)=\area(\Sigma_{u+tf}\cap\pi^{-1}(\Omega))$. It follows from~\eqref{eqn:area-element} and the divergence theorem that
\begin{align}
A_f'(0)&=\int_\Omega\left.\frac{\df}{\df t}\right|_{t=0}\!\!\!\!\sqrt{1+\epsilon\mu^2\|G(u+tf)\|^2}=\int_\Omega\left.\frac{\df}{\df t}\right|_{t=0}\!\!\!\!\sqrt{1+\epsilon\mu^2\|Gu+t\nabla f\|^2}\notag\\
&=\int_\Omega\frac{\epsilon\mu^2\langle Gu,\nabla f\rangle}{\sqrt{1+\epsilon\mu^2\|Gu\|^2}}=-\int_\Omega \epsilon f\diver\left(\frac{\mu^2Gu}{\sqrt{1+\epsilon\mu^2\|Gu\|^2}}\right).\label{eqn:first-variation1}
\end{align}
Moreover, since the associated variational field of this graphical variation is just $\xi$, it is well known (e.g., see~\cite[Lem.~3.1]{BarbosaOliker}) that in both the Riemannian and Lorentzian cases, the first variation of the area functional is also given by
\begin{equation}\label{eqn:first-variation2}
A_f'(0)=-\int_{\Sigma_u\cap\pi^{-1}(\Omega)}2H\langle N,\xi\rangle=-\int_{\Sigma_u\cap\pi^{-1}(\Omega)}\frac{2H\epsilon\mu f}{\sqrt{1+\epsilon\mu^2\|Gu\|^2}}=-\int_\Omega2H\epsilon\mu f.
\end{equation}
Since~\eqref{eqn:first-variation1} and~\eqref{eqn:first-variation2} must agree for all compactly supported functions $f\in\mathcal{C}^\infty_0(U)$, the formula in the statement follows readily.
\end{proof}

\subsection{Working in coordinates}\label{sec:graphs-coordinates}

Consider a Riemannian or Lorentzian Killing submersion $\pi:\mathbb{E}\to M$ where $M$ is a topological disk parametrized globally by $\varphi:\Omega\subset\mathbb{R}^2\to M$ such that the pullback metric of $M$ by $\varphi$ reads $\lambda_1^2\,\df x^2+\lambda_2^2\,\df y^2$ for some positive $\lambda_1,\lambda_2\in\mathcal{C}^\infty(\Omega)$. This can be achieved by taking conformal coordinates, but coordinates which are just orthogonal are much easier to find in practical examples (and lead to the same results).

As the fibers of $\pi$ are assumed to have infinite length, the Killing submersion $\pi$ automatically admits a global section $F_0:M\to\mathbb{E}$, so that
\[\Psi:\Omega\times\mathbb{R}\to\mathbb{E},\quad \Psi((x,y),z)=\phi_z(F_0(\varphi^{-1}(x,y)))\] 
is a global diffeomorphism. Hence, $\mathbb{E}$ can be modeled as $\Omega\times\mathbb{R}$ with the metric that makes $\Psi$ an isometry. The orthonormal frame $\{e_1=\lambda_1^{-1}\partial_x,e_2=\lambda_2^{-1}\partial_y\}$ in $\Omega$ lifts via $\pi$ to a global frame $\{E_1,E_2\}$ of the horizontal distribution, which together with $E_3=\frac{1}{\mu}\partial_z$ forms a global orthonormal frame of $\mathbb{E}$. The fact that the submersion is the projection over the first factor implies that there exist $a,b\in\mathcal{C}^\infty(\Omega)$ such that
\begin{equation}
\begin{aligned}\label{eqn:frame}
(E_1)_{(x,y,z)}&=\tfrac{1}{\lambda_1(x,y)}\,\partial_x+a(x,y)\partial_z,\\
(E_2)_{(x,y,z)}&=\tfrac{1}{\lambda_2(x,y)}\,\partial_y+b(x,y)\partial_z,\\
(E_3)_{(x,y,z)}&=\tfrac{1}{\mu(x,y)}\,\partial_z.
\end{aligned}
\end{equation}
Note that $E_1$ and $E_2$ are spacelike, whereas $E_3$ is spacelike in the Riemannian case and timelike in the Lorentzian case, and $\xi=\partial_z=\mu E_3$ is the Killing vector field. Therefore, the ambient metric in $\mathbb{E}$ can be written as
\begin{equation}\label{eqn:ambient-metric}
\df s^2=\lambda_1^2\df x^2+\lambda_2^2\df y^2+\epsilon\mu^2\left(\df z-\lambda_1a\df x-\lambda_2b\df y\right)^2.
\end{equation}
We will further assume that $\{E_1,E_2,E_3\}$ is positively oriented in $\mathbb{E}$, whence the bundle curvature $\tau$ can be expressed in terms of $a$ and $b$ as
\begin{equation}\label{eqn:tau-model}
\begin{aligned}
\tau&=\tfrac{1}{\mu^2}\langle\nabla_{E_1}\partial_z,E_1\times\partial_z\rangle=-\tfrac{1}{\mu}\langle\nabla_{E_1}\partial_t,E_2\rangle=\langle\nabla_{E_1}E_2,E_3\rangle\\
&=\tfrac{1}{2}\langle[E_1,E_2],E_3\rangle=\tfrac{\epsilon\mu}{2\lambda_1\lambda_2}\left((\lambda_2 b)_x-(\lambda_1 a)_y\right),
\end{aligned}
\end{equation}
where we used the definition of $\tau$ in Equation~\eqref{eqn:tau}, Koszul formula and the following Lie brackets that can be easily deduced from~\eqref{eqn:frame}:
\begin{equation}\label{eqn:bracket}
\begin{aligned}
[E_1,E_2]&=\tfrac{(\lambda_1)_y}{\lambda_1\lambda_2}E_1-\tfrac{(\lambda_2)_x}{\lambda_1\lambda_2}E_2+\tfrac{\mu}{\lambda_1\lambda_2}\left((\lambda_2 b)_x-(\lambda_1 a)_y\right)E_3,\\
[E_1,E_3]&=\tfrac{-\mu_x}{\lambda_1\mu}E_3,\qquad\qquad [E_2,E_3]=\tfrac{-\mu_y}{\lambda_2\mu}E_3.
\end{aligned}\end{equation}

\begin{remark}\label{rmk:calabi-potential}
If $\tau$ and $\mu$ are prescribed, there is a standard way of integrating~\eqref{eqn:tau-model} to obtain $a$ and $b$. Assuming that $\Omega\subset\mathbb{R}^2$ is star-shaped with respect to the origin, the function
\begin{equation} \label{eqn:calabi-potential}
 \mathbf{C}_{M,\tau,\mu}(x,y)
  = 2\int_{0}^{1}  \; s\, \frac{ \tau(sx,sy) \, {\lambda}_{1}(sx,sy) \, {\lambda}_{2}(sx,sy)}
  { {\mu(sx,sy)}} \df s
\end{equation}
will be called the \emph{Calabi potential}. It was originally defined in the unitary case~\cite{Man14,LeeMan} and generalized in~\cite[Eqn.~2.8]{LerMan}. Here we will extend it to consider orthogonal (not necessarily conformal) coordinates in the base. It is straightforward to check that the following choice for $a$ and $b$ satisfies Equation~\eqref{eqn:tau-model}:
\begin{align} \label{eqn:calabi-ab}
a&=\frac{-\epsilon y\, \mathbf{C}_{M,\tau,\mu}}{\lambda_1},&
b&=\frac{\epsilon x\, \mathbf{C}_{M,\tau,\mu}}{\lambda_2}.
\end{align}
Any other pair of functions $a$ and $b$ satisfying~\eqref{eqn:tau-model} produces another isometric metric which is nothing but a change of zero section, see~\cite[pp.~1351--1352]{LerMan}.
\end{remark}

\begin{example}\label{ex:BCV}
In the case of the homogeneous spaces $\mathbb{E}(\kappa,\tau)$ and $\mathbb{L}(\kappa,\tau)$ with isometry group of dimension $4$, we can choose $\lambda_1=\lambda_2$ as the conformal factor $\lambda=(1+\frac{\kappa}{4}(x^2+y^2))^{-1}$ with constant curvature $\kappa$, $\tau$ as a constant and $\mu\equiv 1$. All these functions are defined over the disk $\Omega$ given by the inequality $1+\frac{\kappa}{4}(x^2+y^2)>0$. Then, it is easy to check that~\eqref{eqn:calabi-ab} gives $\mathbf{C}=\tau\lambda$, whence $a=-\epsilon\tau y$ and $b=\epsilon\tau x$. Plugging these functions in~\eqref{eqn:ambient-metric}, we easily recover the classical Cartan model~\cite[\S296]{Cartan} for $\mathbb{E}(\kappa,\tau)$ and $\mathbb{L}(\kappa,\tau)$, see also~\cite{Daniel,Lee11a}.
\end{example}

\begin{example}
Let $A$ be a $2\times 2$ real matrix. We can define the semidirect product $\mathbb{R}^2\ltimes_A\mathbb{R}$ as $\mathbb{R}^3$ endowed with the Lie group structure
\[(p_1,z_1)\star(p_2,z_2)=(p_1+e^{z_1A}p_2,z_1+z_2),\quad (p_1,z_1),(p_2,z_2)\in\mathbb{R}^2\times\mathbb{R},\]
where $e^{zA}=\sum_{k=0}^\infty\frac{z^kA^k}{k!}$ denotes the exponential matrix. Up to isometry, a left-invariant metric in $\mathbb{R}^2\ltimes_A\mathbb{R}$ can be chosen such that the left-invariant frame
\begin{align*}
 E_1&=\alpha_{11}(z)\partial_x+\alpha_{21}(z)\partial_y,&E_2&=\alpha_{12}(z)\partial_x+\alpha_{22}(z)\partial_y,&E_3&=\partial_z,
\end{align*}
is orthonormal (here $\alpha_{ij}(z)$ denote the entries of $e^{zA}$ and $(x,y,z)$ represents the usual coordinates in $\mathbb{R}^3$, see~\cite{MeeksPerez} for a detailed description of these homogeneous spaces). This implies that $\partial_x$ is a Killing field and $\pi(x,y,z)=(y,z)$ is a Killing submersion. The metric can be expressed as
\[\frac{1}{\alpha_{22}^2+\alpha_{21}^2}\, \df y^2+\df z^2+\frac{\alpha_{22}^2+\alpha_{21}^2}{(\alpha_{11}\alpha_{22}-\alpha_{12}\alpha_{21})^2}\left(\df x-\frac{\alpha_{11}\alpha_{21}+\alpha_{12}\alpha_{22}}{\alpha_{22}^2+\alpha_{21}^2}\,\df y\right)^2.\]
By comparing with~\eqref{eqn:ambient-metric}, we can easily identify the functions $\lambda_1,\lambda_2,a,b,\mu$, so that Equation~\eqref{eqn:tau-model} gives the value of the bundle curvature:
\[2\tau=\frac{\alpha_{22}^2+\alpha_{21}^2}{\alpha_{11}\alpha_{22}-\alpha_{12}\alpha_{21}}\left(\frac{\alpha_{11}\alpha_{21}+\alpha_{12}\alpha_{22}}{\alpha_{22}^2+\alpha_{21}^2}\right)_z.\]
This formula was given in~\cite[Ex.~2.4]{LerMan} with a mistake, which is now fixed.
\end{example}

We will now describe how to compute the mean curvature of a graph over an open subset $U\subset M$ in coordinates. We can choose the zero section $F_0:U\to\mathbb{E}$ as $F_0(x,y)=(x,y,0)$, so a graph parametrized by $u\in\mathcal{C}^\infty(U)$ can be expressed as $F_u(x,y)=(x,y,u(x,y))$. This also gives rise to the distance along vertical fibers $d(x,y,z)=z$. Taking into account~\eqref{eqn:frame}, we can work out the gradient
\[\overline\nabla d=E_1(z)E_1+E_2(z)E_2+\epsilon E_3(z)E_3=aE_1+bE_2+\frac{\epsilon}{\mu}E_3,\]
so that $Z=\pi_*(\overline\nabla d)=ae_1+be_2$ and~\eqref{eqn:tau-model} yields
\begin{equation}\label{eqn:JZ}
\diver(JZ)=\diver(-be_1+ae_2)=\frac{-1}{\lambda_1\lambda_2}\left((\lambda_2 b)_x-(\lambda_1 a)_y\right)=\frac{-2\epsilon\tau}{\mu},
\end{equation}
so that $Z$ encodes information about the bundle curvature. Note also that 
\begin{equation}\label{eqn:alpha-beta}
Gu=\alpha e_1+\beta e_2,\qquad\text{where }\alpha=\frac{u_x}{\lambda_1}-a\text{ and } \beta=\frac{u_y}{\lambda_2}-b.
\end{equation}
Therefore, the equation for the mean curvature given by Proposition~\ref{prop:mean-curvature} can be written in coordinates as
\begin{equation}\label{eqn:H}
2 H=  \frac{1}{\mu\lambda_1\lambda_2}\left[\frac{\partial}{\partial x} \left(\mu^2\frac{\lambda_2\alpha}{\omega}\right) +\frac{\partial}{\partial y}  \left(\mu^2\frac{\lambda_1\beta}{\omega}\; \right) \; \right],
\end{equation}
where $\omega=\sqrt{1+\epsilon\mu^2\|Gu\|^2}=\sqrt{1+\epsilon\mu^2(\alpha^2+\beta^2)}$ is the area element we found in~\eqref{eqn:area-element}. Notice that the spacelike condition in the Lorentzian case ($\epsilon=-1$) can be written as $\alpha^2+\beta^2<\mu^{-2}$.

The standard frame $\{\partial_x,\partial_y\}$ in $M$ can be lifted via $\pi$ to the tangent frame $\{X=\lambda_1(E_1+\mu\alpha E_3),Y=\lambda_2(E_2+\mu\beta E_3)\}$ in $\Sigma_u$, whence
\begin{align*}
\langle X,X\rangle&=\lambda_1^2(1+\epsilon\mu^2\alpha^2),&
\langle X,Y\rangle&=\epsilon\lambda_1\lambda_2\mu^2\alpha\beta,&
\langle Y,Y\rangle&=\lambda_2^2(1+\epsilon\mu^2\beta^2).
\end{align*}
Therefore, $\pi|_{\Sigma_u}:\Sigma_u\to U$ induces the following Riemannian metric in $U\subset M$:
\begin{equation}\label{eqn:induced-metric}
\lambda_1^2(1+\epsilon\mu^2\alpha^2)\df x^2+2\epsilon\lambda_1\lambda_2\mu^2\alpha\beta\df x\df y+\lambda_2^2(1+\epsilon\mu^2\beta^2)\df y^2.
\end{equation}

\subsection{The curvature tensor}
Our next goal is to compute the Riemann curvature tensor of the total space of a Killing submersion $\pi:\mathbb{E}\to M$ to understand its geometry. Since the computation is local, we will employ the coordinates we have introduced in Section~\ref{sec:graphs-coordinates}, where $\mathbb{E}$ is (locally) identified with $\Omega\times\mathbb{R}$ for some $\Omega\subseteq\mathbb{R}^2$ with the metric in~\eqref{eqn:ambient-metric} for some positive functions $\lambda_1,\lambda_2,\mu\in\mathcal{C}^\infty(\Omega)$ and arbitrary functions $a,b\in\mathcal{C}^\infty(\Omega)$. Using~\eqref{eqn:tau-model},~\eqref{eqn:bracket} and Koszul formula, we can write the Levi-Civita connection $\overline\nabla$ of $\mathbb{E}$ in the frame $\{E_1,E_2,E_3\}$ given by~\eqref{eqn:frame}:
\begin{align*}
\overline\nabla_{E_1}E_1&=-\tfrac{(\lambda_1)_y}{\lambda_1\lambda_2}E_2,&
\overline\nabla_{E_1}E_2&=\tfrac{(\lambda_1)_y}{\lambda_1\lambda_2}E_1+\epsilon\tau E_3,&
\overline\nabla_{E_1}E_3&=-\tau E_2,\\
\overline\nabla_{E_2}E_1&=\tfrac{(\lambda_2)_x}{\lambda_1\lambda_2}E_2-\epsilon\tau E_3,&
\overline\nabla_{E_2}E_2&=-\tfrac{(\lambda_2)_x}{\lambda_1\lambda_2}E_1,&
\overline\nabla_{E_2}E_3&=\tau E_1,\\
\overline\nabla_{E_3}E_1&=-\tau E_2+\tfrac{\mu_x}{\lambda_1\mu}E_3,&
\overline\nabla_{E_3}E_2&=\tau E_1+\tfrac{\mu_y}{\lambda_2\mu}E_3,&
\overline\nabla_{E_3}E_3&=-\tfrac{\epsilon\mu_x}{\lambda_1\mu}E_1-\tfrac{\epsilon\mu_y}{\lambda_2\mu}E_2.
\end{align*}
Therefore, we can work out $\overline R(X,Y)Z=\overline\nabla_X\overline\nabla_Y Z-\overline\nabla_Y\overline\nabla_XZ-\overline\nabla_{[X,Y]} Z$, the three-variable Riemann curvature tensor, over this frame to obtain
\begin{align*}
R(E_1,E_2)E_1&=-(K_M-3\epsilon\tau^2)E_2-\epsilon\langle T,E_1\rangle E_3,\\
R(E_1,E_2)E_2&=(K_M-3\epsilon\tau^2)E_1-\epsilon\langle T,E_2\rangle E_3,\\
R(E_1,E_2)E_3&=\langle T,E_1\rangle E_1+\langle T,E_2\rangle E_2,\\
R(E_1,E_3)E_1&=-\langle T,E_1\rangle E_2-(\epsilon\tau^2-a_{11})E_3,\\
R(E_1,E_3)E_2&=\langle T,E_1\rangle E_1+a_{12}E_3,\\
R(E_1,E_3)E_3&=(\tau^2-\epsilon a_{11})E_1-\epsilon a_{12}E_2,\\
R(E_2,E_3)E_1&=-\langle T,E_2\rangle E_2+a_{21} E_3,\\
R(E_2,E_3)E_2&=\langle T,E_2\rangle E_1-(\epsilon\tau^2-a_{22})E_3,\\
R(E_2,E_3)E_3&=-\epsilon a_{21}E_1+(\tau^2-\epsilon a_{22})E_2,
\end{align*}
where $T=\overline\nabla\tau+\tfrac{2\tau}{\mu}\overline\nabla\mu$ and $a_{ij}=\frac{1}{\mu}\overline{\mathrm{Hess}}(\mu)(E_i,E_j)$. Here, the Hessian is defined by $\overline{\mathrm{Hess}}(\mu)(X,Y)=X(Y(\mu))-(\overline\nabla_XY)(\mu)$ for all vector fields $X$ and $Y$ in $\mathbb{E}$. These coefficients $a_{ij}$ are explicitly given by
\begin{align*}
a_{11}&=\tfrac{1}{\mu}E_1(E_1(\mu))+\tfrac{1}{\lambda_1\mu}E_2(\lambda_1)E_2(\mu),& 
a_{12}&=\tfrac{1}{\mu}E_1(E_2(\mu))-\tfrac{1}{\lambda_1\mu}E_2(\lambda_1)E_1(\mu),\\
a_{21}&=\tfrac{1}{\mu}E_2(E_1(\mu))-\tfrac{1}{\lambda_2\mu}E_1(\lambda_2)E_2(\mu),&
a_{22}&=\tfrac{1}{\mu}E_2(E_2(\mu))+\tfrac{1}{\lambda_2\mu}E_1(\lambda_2)E_1(\mu).
\end{align*}
Recall that $a_{12}=a_{21}$ by the symmetry of the Hessian. Also in the above computations, we have introduced the Gauss curvature of $M$ given by
\[K_M=\frac{(\lambda_1)_x(\lambda_2)_x\lambda_2^2+(\lambda_1)_y(\lambda_2)_y\lambda_1^2}{\lambda_1^3\lambda_2^3}-\frac{(\lambda_2)_{xx}\lambda_2+(\lambda_1)_{yy}\lambda_1}{\lambda_1^2\lambda_2^2}.\]
The four-variable Riemann curvature tensor $\overline R(X,Y,Z,W)=\langle\overline R(X,Y)Z,W\rangle$ can be computed coordinate-freely as follows.

\begin{proposition}\label{prop:curvature}
If $X,Y,Z,W$ are vector fields in $\mathbb{E}$, then
\begin{align*}
\overline R(X,Y,Z,W)&=-\tau^2\langle X\times Y,Z\times W\rangle-(K_M-4\epsilon\tau^2)\langle X\times Y,E_3\rangle\langle Z\times W,E_3\rangle\\
&\quad+\langle X\times Y,E_3\rangle\langle Z\times W,E_3\times T\rangle+\langle Z\times W,E_3\rangle\langle X\times Y,E_3\times T\rangle\\
&\quad+\tfrac{\epsilon}{\mu}\overline{\mathrm{Hess}}(\mu)((X\times Y)\times E_3,(Z\times W)\times E_3).
\end{align*}
In particular, the sectional curvature of a spacelike plane $\Pi\subset T_p\mathbb E$ is given by
\begin{align*}
\overline K(\Pi)&=\epsilon\tau^2+(K_M-4\epsilon\tau^2)\langle n,E_3\rangle^2-2\langle n,E_3\rangle\langle n\times E_3,T\rangle\\
&\qquad-\tfrac{\epsilon}{\mu}\overline{\mathrm{Hess}}(\mu)(n\times E_3,n\times E_3),
\end{align*}
where $n\in T_p\mathbb E$ is a unit normal to $\Pi$.
\end{proposition}

\begin{proof}
It suffices to check that both sides coincide on the frame $\{E_1,E_2,E_3\}$, which is a straightforward computation. It is important to notice first that $E_1\times E_2=\epsilon E_3$, $E_2\times E_3=E_1$ and $E_3\times E_1=E_2$ by definition of cross product. As for the sectional curvature, we choose an orthonormal basis $\{u,v\}$ of $\Pi$ such that $u\times v=n$ and then compute $\overline K(\Pi)=\overline R(u,v,v,u)$ taking into account that $\langle n,n\rangle=\epsilon$.
\end{proof}

\begin{remark}
The structure of the expression for $\overline{R}(X,Y,Z,W)$ is meaningful. The first summand is the curvature of a space form of constant curvature since $\langle X\times Y,Z\times W\rangle=\langle X,Z\rangle\langle Y,W\rangle-\langle Y,Z\rangle\langle X,W\rangle$. The second summand shows up in homogeneous spaces $\mathbb{E}(\kappa,\tau)$ and $\mathbb{L}(\kappa,\tau)$ with four-dimensional isometry group for the standard submersion over $\mathbb{M}^2(\kappa)$. The next two summands appear in arbitrary Killing submersion with unitary Killing vector field (see also~\cite[Lem.~5.1]{Man14}). The last summand containing the Hessian only appears if the Killing vector field has non-constant length.
\end{remark}

\section{The conformal Calabi-type duality}\label{sec:duality}

In this section, we will connect the family of entire graphs in Riemannian Killing submersions and the family of entire spacelike graphs in Lorentzian Killing submersions with the only assumption that the common base surface is simply connected. Since we need to work simultaneously with two surfaces, we will use a tilde in the Lorentzian setting, which means that the corresponding element must be computed by means of the background in Section~\ref{sec:preliminaries} with $\epsilon=-1$.

The classical Calabi duality~\cite{Calabi} relies on the fact that the minimal (resp.\ maximal) graph equation in $\mathbb{R}^3=\mathbb{E}(\mathbb{R}^2,0,1)$ (resp.\ $\mathbb{L}^3=\mathbb{L}(\mathbb{R}^2,0,1)$) is a divergence-zero equation in $\mathbb{R}^2$. In general, in a simply connected base surface $M$, the Poincaré lemma says that a divergence-zero equation $\diver(X)=0$ implies the existence of a function $f$ in $M$ such that $X=J\nabla f$, where $J$ is a $\frac{\pi}{2}$-rotation in the tangent bundle of $M$. In~\cite{Lee11a,LeeMan}, the first and third authors managed to produce divergence zero equations even if the mean curvature is not zero. In the next result, we show that there is a natural and geometric way to get rid of the mean curvature function by means of the vector field $Z$ defined in Section~\ref{sec:preliminaries}.

\begin{theorem}[Conformal duality]\label{thm:duality}
Let $M$ be a simply connected Riemannian surface and let $\tau,H,\mu\in\mathcal{C}^\infty(M)$ be arbitrary functions such that $\mu>0$. There is a bijective correspondence between
\begin{enumerate}[label=\emph{(\alph*)}]
  \item entire graphs in $\mathbb{E}(M,\tau,\mu)$ with prescribed mean curvature $H$, and
  \item entire graphs in $\mathbb{L}(M,H,\mu^{-1})$ with prescribed mean curvature $\tau$.
\end{enumerate}
Assume that $\Sigma\subset \mathbb{E}(M,\tau,\mu)$ and $\widetilde\Sigma\subset\mathbb{L}(M,H,\mu^{-1})$ are such corresponding graphs.
\begin{enumerate}[label=\emph{(\arabic*)}]
  \item The graphs $\Sigma$ and $\widetilde\Sigma$ determine each other up to vertical translations. 
  \item The corresponding angle functions $\nu,\widetilde\nu:M\to\mathbb{R}$ satisfy $\widetilde\nu=-\nu^{-1}$.
  \item The diffeomorphism $\Phi:\Sigma\to\widetilde\Sigma$ such that $\widetilde\pi\circ\Phi=\pi$, where $\pi:\mathbb{E}(M,\tau,\mu)\to M$ and $\widetilde\pi:\mathbb{L}(M,H,\mu^{-1})\to M$ are the involved Killing submersions, is conformal with conformal factor $\Phi^*\df s_{\widetilde\Sigma}^2=\mu^{-2}\nu^2\df s_\Sigma^2$.
  
\end{enumerate}
Moreover, both families \emph{(a)} and \emph{(b)} are empty if $M$ is a topological sphere and either $\int_M\frac\tau\mu\neq 0$ or $\int_M H\mu\neq 0$.
\end{theorem}

\begin{proof}
If $M$ is a topological sphere and $\int_M\frac\tau\mu\neq 0$, then the Killing submersion $\pi:\mathbb{E}(M,\tau,\mu)\to M$ is the Hopf fibration~\cite[Thm~2.9]{LerMan}, which admits no entire sections. Also, there is no entire graph with prescribed mean curvature $\tau$ in $\mathbb{L}(M,H,\mu^{-1})$, because such a graph would produce a smooth field $X$ on $M$ such that $\diver(X)=\frac{\tau}{2\mu}$, whence $\int_M\frac\tau\mu=0$ by the divergence theorem. This means that both families in (a) and (b) are empty if $M$ is a topological sphere and $\int_M\frac\tau\mu\neq 0$. Analogously, both are empty if $M$ is a topological sphere and $\int_M H\mu\neq 0$.

Therefore, we can assume that there are global sections $F_0:M\to\mathbb{E}(M,\tau,\mu)$ and $\widetilde F_0:M\to \mathbb{L}(M,H,\mu^{-1})$, see~\cite[Prop.~3.3]{LerMan}. These sections produce smooth vector fields $Z,\widetilde Z\in\mathfrak{X}(M)$ such that $\diver(JZ)=\frac{-2\tau}{\mu}$ and $\diver(J\widetilde Z)=2H\mu$. Let $u\in\mathcal{C}^\infty(M)$ whose graph over the zero section $F_0$ has prescribed mean curvature $H$, that is,
\begin{equation}\label{thm:duality:eqn1}
2H\mu=\diver\left(\frac{\mu^2 Gu}{\sqrt{1+\mu^2\|Gu\|^2}}\right)=\diver(J\widetilde Z).
\end{equation}
Since $M$ is simply connected and~\eqref{thm:duality:eqn1} can be written as a divergence zero equation, the Poincaré lemma yields the existence of $v\in\mathcal{C}^\infty(M)$ such that 
\begin{equation}\label{thm:duality:eqn2}
\frac{\mu^2 Gu}{\sqrt{1+\mu^2\|Gu\|^2}}-J\widetilde Z=-J\nabla v\ \Leftrightarrow\ \frac{\mu^2 Gu}{\sqrt{1+\mu^2\|Gu\|^2}}=-J\widetilde G v,
\end{equation}
where $\widetilde Gv=\nabla v-\widetilde Z$ is the generalized gradient in $\mathbb L(M,H,\mu^{-1})$. The function $v$ is univocally determined up to an additive constant, which proves item (1) in the statement. Taking square norms in~\eqref{thm:duality:eqn2}, we find that
\begin{equation}\label{thm:duality:eqn3}
\frac{\mu^4\|Gu\|^2}{1+\mu^2\|Gu\|^2}=\|\widetilde G v\|^2\ \Leftrightarrow\ \frac{1}{1+\mu^2\|Gu\|^2}=1-\mu^{-2}\|\widetilde G v\|^2.
\end{equation}
The right-hand side in~\eqref{thm:duality:eqn3} reveals that $1-\mu^{-2}\|\widetilde Gv\|^2>0$, whence the graph defined by $v$ over the zero section $\widetilde F_0$ is spacelike. Taking into account~\eqref{eqn:angle-function} and~\eqref{thm:duality:eqn3}, we easily reach item (2). Also, we can plug~\eqref{thm:duality:eqn3} into~\eqref{thm:duality:eqn2} to get
\begin{equation}\label{thm:duality:eqn4}
\diver\left(\frac{\mu^{-2}\widetilde G v}{\sqrt{1+\mu^{-2}\|\widetilde Gv\|^2}}\right)=\diver(JGu)=\diver(J\nabla u)-\diver(JZ)=\frac{2\tau}{\mu},
\end{equation}
so the graph defined by $v$ has mean curvature $\tau$ in $\mathbb{L}(M,H,\mu^{-1})$. Likewise, we can obtain a graph in $\mathbb{E}(M,\tau,\mu)$ with mean curvature $H$ starting with a spacelike graph in $\mathbb{L}(M,H,\mu^{-1})$ with mean curvature $\tau$, so the duality is a bijection.

It remains to check item (3) to finish the proof. It suffices to check that the metrics induced by $\pi$ and $\widetilde\pi$ in $M$ differ in the desired conformal factor. Since this property is local, we will work in coordinates using the background described in Section~\ref{sec:graphs-coordinates}, where $M=(\Omega,\lambda_1^2\df x^2+\lambda_2\df y^2)$ with $\Omega\subset\mathbb{R}^2$. Equation~\eqref{eqn:alpha-beta} says that we can express $Gu=\alpha e_1+\beta e_2$ and $\widetilde Gv=\widetilde\alpha e_1+\widetilde\beta e_2$, where $\alpha=\frac{u_x}{\lambda_1}-a$, $\beta=\frac{u_y}{\lambda_2}-b$, $\widetilde\alpha=\frac{v_x}{\lambda_1}-\widetilde a$ and $\widetilde\beta=\frac{v_y}{\lambda_2}-\widetilde b$. If we consider the area elements 
\[\omega=\sqrt{1+\mu^2(\alpha^2+\beta^2)},\qquad 
\widetilde\omega=\sqrt{1-\mu^{-2}(\widetilde{\alpha}^2+\widetilde{\beta}^2)},\]
then~\eqref{thm:duality:eqn3} implies that $\omega\widetilde\omega=1$, whence~\eqref{thm:duality:eqn2} can be written in two equivalent ways:
\begin{equation}\label{eqn:twin}\left(\widetilde\alpha,\widetilde\beta\right)=\left(\frac{-\mu^2\beta}{\omega},\frac{\mu^2\alpha}{\omega}\right)\ \Leftrightarrow\ \left(\alpha,\beta\right)=\left(\frac{\widetilde\beta}{\mu^2\widetilde\omega},\frac{-\widetilde\alpha}{\mu^2\widetilde\omega}\right).
\end{equation}
These \emph{twin relations} allow us to compute
\begin{align*}
\lambda_1^2\left(1-\frac{\widetilde\alpha^2}{\mu^2}\right)&=\lambda_1^2\left(1-\frac{\mu^2\beta^2}{\omega^2}\right)=\frac{\lambda_1^2(1+\mu^2\alpha^2)}{\omega^2},\\
-\lambda_1\lambda_2\frac{\widetilde\alpha\widetilde\beta}{\mu^2}&=\frac{\lambda_1\lambda_2\mu^2\alpha\beta}{\omega^2},\\
\lambda_2^2\left(1-\frac{\widetilde\beta^2}{\mu^2}\right)&=\lambda_2^2\left(1-\frac{\mu^2\alpha^2}{\omega^2}\right)=\frac{\lambda_2^2(1+\mu^2\beta^2)}{\omega^2}.
\end{align*}
Taking into account the expression~\eqref{eqn:induced-metric} for the induced metrics in $M$, we deduce that both metrics are conformal with conformal factor $\omega^{-2}=\mu^{-2}\nu^2$.
\end{proof}

\begin{remark}
In coordinates, we only need to choose the functions $a,b,\widetilde a,\widetilde b$ giving the desired bundle curvatures (which amounts to choosing the initial section). Once this is achieved, the twin relations~\eqref{eqn:twin} actually give a first-order \textsc{ode} system
\begin{equation}\label{eqn:twin-v}\left(\widetilde\alpha,\widetilde\beta\right)=\left(\frac{-\mu^2\beta}{\omega},\frac{\mu^2\alpha}{\omega}\right)\ \Leftrightarrow\ \begin{cases}
v_x=\lambda_1\widetilde a+\frac{-\lambda_1\mu(\frac{u_y}{\lambda_2}-b)}{\sqrt{\mu^{-2}+(\frac{u_x}{\lambda_1}-a)^2+(\frac{u_y}{\lambda_2}-b)^2}},\\[3pt]
v_y=\lambda_2\widetilde b+\frac{\lambda_2\mu(\frac{u_x}{\lambda_1}-a)}{\sqrt{\mu^{-2}+(\frac{u_x}{\lambda_1}-a)^2+(\frac{u_y}{\lambda_2}-b)^2}}.
\end{cases}
\end{equation}
Equivalently,
\begin{equation}\label{eqn:twin-u}\left(\alpha,\beta\right)=\left(\frac{\widetilde\beta}{\mu^2\widetilde\omega},\frac{-\widetilde\alpha}{\mu^2\widetilde\omega}\right)\ \Leftrightarrow\ \begin{cases}
u_x=\lambda_1a+\frac{\frac{\lambda_1}{\mu}(\frac{v_y}{\lambda_2}-\widetilde b)}{\sqrt{\mu^2-(\frac{v_x}{\lambda_1}-\widetilde a)^2-(\frac{v_y}{\lambda_2}-\widetilde b)^2}},\\[3pt]
u_y=\lambda_2b+\frac{-\frac{\lambda_2}{\mu}(\frac{v_x}{\lambda_1}-\widetilde a)}{\sqrt{\mu^2-(\frac{v_x}{\lambda_1}-\widetilde a)^2-(\frac{v_y}{\lambda_2}-\widetilde b)^2}}.
\end{cases}
\end{equation}
The prescribed mean curvature $H$ or $\tau$ equation in $\mathbb{E}(M,\tau,\mu)$ or $\mathbb{L}(M,H,\mu^{-1})$, respectively, can be now thought of as the compatibility conditions for these systems. 
\begin{itemize}
	\item The classical Calabi duality~\cite{Calabi} is recovered for $a=b=\widetilde a=\widetilde b=0$ and $\mu=\lambda_1=\lambda_2=1$ (so we get the flat base surface $M=\mathbb{R}^2$).
	\item The duality in homogeneous spaces with four-dimensional isometry group is recovered for $\lambda_1=\lambda_2=(1+\tfrac{\kappa}{4}(x^2+y^2))^{-1}$, $a=-\tau y$, $b=\tau x$, $\widetilde a=H y$, $\widetilde b=-H x$, and $\mu=1$, see~\cite[Cor.~2]{Lee11a}. In this case, we have the base surface $M=\mathbb{M}^2(\kappa)$ (minus a point if $\kappa>0$) as explained in Example~\ref{ex:BCV}.
\end{itemize}
\end{remark} 

\section{Entire graphs with prescribed mean curvature in $\mathbb{L}^3$}\label{sec:L3}

Let $H\in\mathcal{C}^\infty(\mathbb{R}^2)$ be a smooth function. We would like to obtain an entire spacelike graph $z=v(x,y)$ in $\mathbb{L}^3=\mathbb{L}(\mathbb{R}^2,0,1)$ whose mean curvature at the point $(x,y,v(x,y))$ is precisely $H(x,y)$ for all $(x,y)\in\mathbb{R}^2$. By the duality in Theorem~\ref{thm:duality} (indeed, it suffices to apply the duality in the unitary case, see~\cite{LeeMan}), this is equivalent to finding an entire minimal graph in $\mathbb{R}^3_H=\mathbb{E}(\mathbb{R}^2,H,1)$. We will need a couple of lemmas to prove the existence of such an entire minimal graph, though we will need that both $H$ and $\nabla H$ are bounded in order to apply the following result.

\begin{lemma}\label{lem:halfspace}
If $H$ and $\nabla H$ are bounded, then there is no properly immersed surfaces in a connected component of $\mathbb{R}^3_H-P$, where $P$ is a vertical plane.
\end{lemma}

\begin{proof}
Using the Calabi potential (see Remark~\ref{rmk:calabi-potential}), the manifold $\mathbb{R}^3_H$ can be modeled as $\mathbb{R}^3$ endowed with the Riemannian metric
\[\df x^2+\df y^2+(\df z+y\,\mathbf{C}\df x-x\,\mathbf{C}\df y)^2,\qquad\text{where }\mathbf{C}(x,y)=2\int_0^1sH(xs,ys)\df s,\]
where we can also assume (after an \emph{a priori} rotation) that $P$ is given $y=d$ for some $d\in\mathbb{R}$. Consider the foliation by planes $P_t=\{(x,y,z)\in\mathbb{R}^3:y=t\}$, in which $P_d=P$ and each $P_t$ is flat and minimal since it projects onto a geodesic of $\mathbb{R}^2$. In particular, each leave $P_t$ is a parabolic surface. Observe that $E_1=\partial_x-y\mathbf{C}\partial_z$ and $E_3=\partial_z$ form an orthonormal tangent frame to all $P_t$ in which we can compute the second fundamental form as
\[\sigma_t\equiv\begin{pmatrix}
\sigma_t(E_1,E_1)&\sigma_t(E_1,E_3)\\\sigma_t(E_3,E_3)&\sigma_t(E_3,E_3)
\end{pmatrix}=\begin{pmatrix}
0&H\\H&0
\end{pmatrix}.\]
This computation is essentially the same as in~\cite[p.~1361]{LerMan} taking into account that $\mu\equiv 1$ and $P_t$ projects onto a geodesic. Therefore, $\|\sigma_t\|^2=2H^2$ is uniformly bounded not depending on $t$. Finally, consider the projection $\Phi_t:P_t\to P_0$ sending $(x,d+t,z)$ to $(x,d,z)$. Its differential $\df\Phi_t$ sends the orthonormal frame $\{E_1,E_3\}$ in $P_t$ to the frame $\{E_1-t\mathbf{C}E_3,E_3\}$ in $P_0$. However, since $H$ is bounded, so is $C$ and it trivially follows that $\Phi_t$ is a quasi-isometric projection into $P_0$ when $t$ is close to $d$. Proposition~\ref{prop:curvature} yields the following bound for the sectional curvature of $\mathbb{R}^3_H$:
\[|\overline K(\Pi)|=\left|H^2-4H^2\langle n,E_3\rangle^2-2\langle n,E_3\rangle\langle n\times E_3,\nabla H\rangle\right|\leq 3H^2+2\|\nabla H\|^2.\]
Since $H$ and $\nabla H$ are bounded, the geometry of $\mathbb{R}^3_H$ is bounded. All the hypothesis of the halfspace theorem in~\cite[Thm.~7]{Mazet} are met, so we deduce that there are no properly immersed surfaces in a connected component of $\mathbb{R}^3_H-P$.
\end{proof}

\begin{lemma}\label{lem:angle-1}
For each $r>0$, there is a minimal graph in $\mathbb{R}^3_H$ over $D_r=\{(x,y)\in\mathbb{R}^2:x^2+y^2<r^2\}$ with angle function equal to $1$ at $(0,0)$.
\end{lemma}

\begin{proof}
Let $\mathbb{S}^2_+=\{\varphi\in\mathbb{R}^3:\varphi_1^2+\varphi_2^2+\varphi_3^2=1,\ \varphi_3> 0\}$ be the open upper halfsphere in $\mathbb{R}^3$. For each $\varphi\in\mathbb{S}^2_+$ with $(\varphi_1,\varphi_2)\neq(0,0)$, decompose $\partial D_r=S_\varphi^+\cup S_\varphi^-$, where 
\begin{align*}
S_\varphi^+&=\{(x,y)\in\partial D_r:\langle(x,y),(\varphi_1,\varphi_2)\rangle>0\},\\ 
S_\varphi^-&=\{(x,y)\in\partial D_r:\langle(x,y),(\varphi_1,\varphi_2)\rangle<0\},
\end{align*}
and consider the boundary data in $\partial D_r$ that assigns a value $\pm(\varphi_3^{-2}-1)$ to the component $S_\varphi^\pm$, see Figure~\ref{fig:degree}. If $\varphi_1=\varphi_2=0$, the value $0$ is assigned to all $\partial D_r$. Let $\Sigma_\varphi\subset\mathbb{R}^3_H$ be the minimal graph over $\overline D_r$ that solves the Dirichlet problem for such boundary data. Note that such a minimal surface exists and is unique by~\cite[Thm.~4.4]{DMN}. The uniqueness also guarantees that $\Sigma_\varphi$ depends continuously on $\varphi$ since the boundary data we have defined in turn depend continuously on $\varphi$. Additionally, we define $\Sigma_\varphi\subset\mathbb{R}^3_H$ as the minimal vertical plane with normal $\varphi_1\partial_x+\varphi_2\partial_y$ at the origin whenever $\varphi_1^2+\varphi_2^2=1$ and $\varphi_3=0$. Recall that $\{\partial_x,\partial_y,\partial_z\}$ is an orthonormal basis of $\mathbb{R}_H^3$ at the origin in our model.

\begin{figure}
\includegraphics[width=0.6\textwidth]{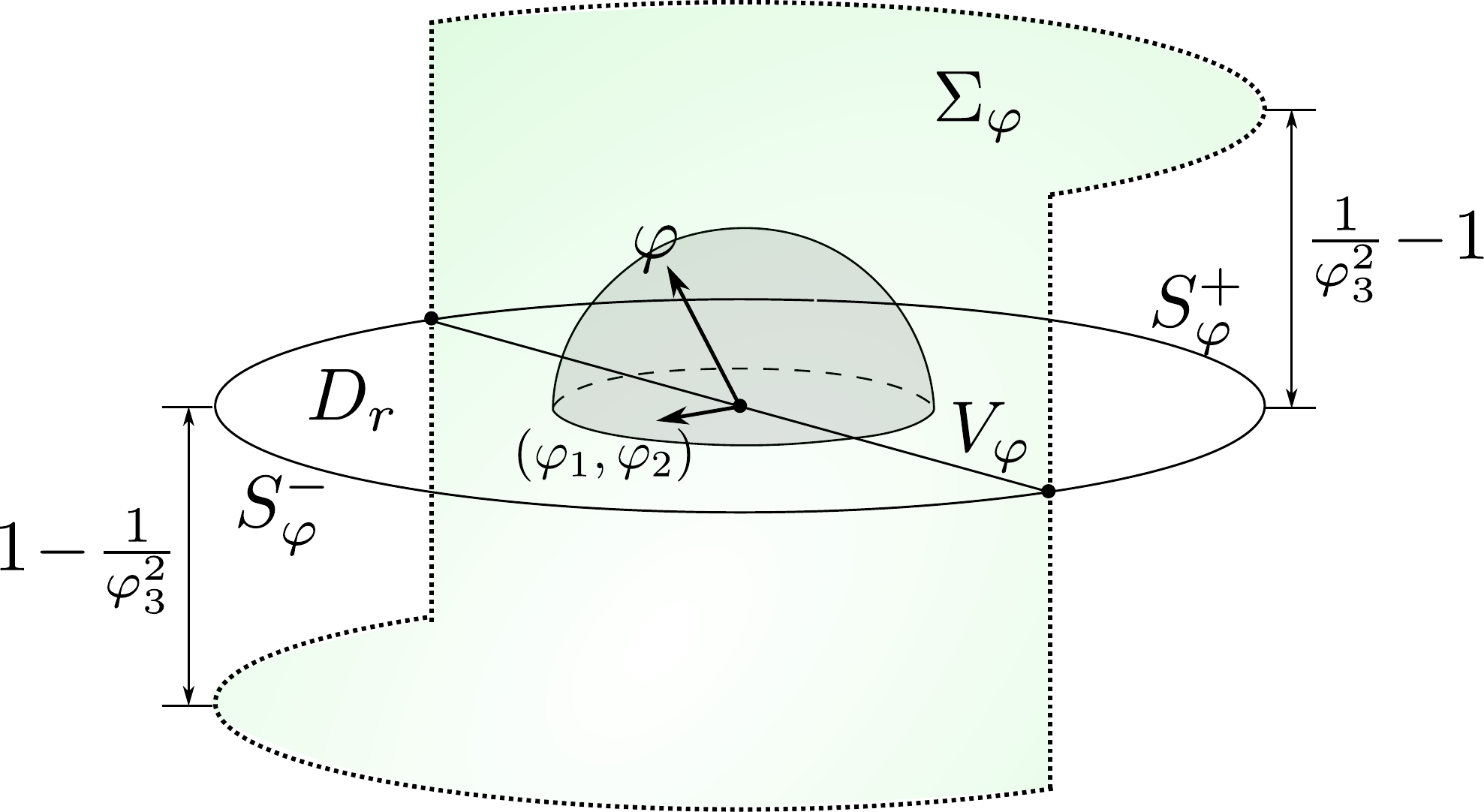}
\caption{The green surface $\Sigma_\varphi$ that solves the Dirichlet problem over $D_r$ with boundary values $\pm(\varphi_3^{-2}-1)$ on the half-circles $S_\varphi^{\pm}$.}\label{fig:degree}
\end{figure}

This allows us to define a map $\eta:\overline{\mathbb{S}}{}^2_+\to\overline{\mathbb{S}}{}^2_+$ in the closed upper hemisphere such that $\eta(\varphi)=\psi$ if the unit normal of $\Sigma_\varphi$ is expressed as $\psi_1\partial_x+\psi_2\partial_y+\psi_3\partial_z$. This map is continuous on $\mathbb S^2_+$, the interior of the hemisphere, by the continuity of $\Sigma_\varphi$ with respect to $\varphi$, but it is also continuous in the closure. To see this, for each $\varphi\in\mathbb{S}^2_+$ with $\varphi_3\neq 1$, let $V_\varphi$ the diameter of $D_r$ joining the endpoints of the arc $S_\varphi^+$ and let $\Sigma_\varphi^\pm\subset\mathbb{R}^3_H$ be the minimal graph that solves the Jenkins--Serrin problem over the halfdisk demarcated by $S_\varphi^\pm$ and $V_\varphi$ with boundary values $\pm(\varphi_3^{-2}-1)$ on $S_\varphi^\pm$ and $\mp\infty$ on $V_\varphi$, which exists by~\cite[Thm.~6.4]{DMN}. By the maximum principle (see~\cite[Thm.~6.8]{DMN}), the surface $\Sigma_\varphi$ lies below $\Sigma_\varphi^+$ and above $\Sigma_\varphi^-$ as graphs.

Given $\varphi_0\in\partial\mathbb{S}^2_+$, the radial limit of $\Sigma_\varphi$ as $\varphi\to\varphi_0$ is the vertical plane $\Sigma_{\varphi_0}=V_\varphi\times\mathbb{R}$ because $S_\varphi^+$ and $S_\varphi^-$ sweep out the whole region outside this plane as $\varphi\to\varphi_0$ radially ($V_\varphi$ does not change in the radial limit). This also means that $\eta(\varphi_0)=\varphi_0$. Since the radial limit of $\eta$ is continuous in all $\partial\mathbb{S}^2_+$, we infer that $\eta:\overline{\mathbb{S}}{}^2_+\to\overline{\mathbb{S}}{}^2_+$ is continuous. Since $\eta(\varphi)=\varphi$ for all $\varphi\in\partial\mathbb{S}^2_+$, an easy degree argument shows that $\eta$ is onto, whence there is some $\varphi_0\in\mathbb{S}^2_+$ such that $\eta(\varphi_0)=(0,0,1)$ so that $\Sigma_{\varphi_0}$ is the desired minimal graph over $D_r$.
\end{proof}

\begin{theorem}\label{thm:L3}
If $H\in\mathcal{C}^\infty(\mathbb{R}^2)$ is a bounded function such that $\nabla H$ is also bounded, then there is an entire spacelike graph in $\mathbb{L}^3$ with prescribed mean curvature $H$.
\end{theorem}

\begin{proof}
For each $n\in\mathbb{N}$, let $\Sigma_n\subset\mathbb{R}^3_H$ be a minimal graph over the disk $D_n$ passing through the origin with angle function $1$ given by Lemma~\ref{lem:angle-1}. For any $r>0$, since $D_r$ is relatively compact, the geometry of $D_r\times\mathbb{R}\subset\mathbb{R}^3_H$ is uniformly bounded and the surfaces $\Sigma_n\cap(D_r\times\mathbb{R})$ are graphs over $D_r$, so they are stable and their second fundamental forms are uniformly bounded for $n\geq r+1$. Since they have the origin as accumulation point, we can find a subsequence $\Sigma_n'$ of $\Sigma_n$ that converges to a minimal surface $\Sigma_{r,\infty}$ containing the origin. This surface $\Sigma_{r,\infty}$ is either a minimal graph over a neighborhood of the origin of $\mathbb{R}^2$ or a vertical plane, but the latter can be discarded since the angle functions are constant $1$ at the origin. If $\Sigma_{r,\infty}$ is not a graph over all $D_r$ is because the sequence $\Sigma_n$ has divergence lines. We recall that a divergence line is a straight segment $L\subset D_r$ connecting two points of $\partial D_r$ such that vertical translations of $\Sigma_n'$ subconverge on compact subsets to $\pi^{-1}(L)$, see~\cite[\S7]{DMN}. We will show that we can find a further subsequence of $\Sigma_n$ such that these divergence lines have no intersection in $D_r$. To this end, we start with a disk $D_\rho$ of maximal radius where the sequence $\Sigma_n$ subconverges, and assume that $\rho\leq r$, which means that we can find some divergence lines touching $\partial D_\rho$. Choose one, say $L$, and take a further subsequence of $\Sigma_n$ such that $L$ does not intersect any other divergence line inside $D_r$ (this is a standard argument, see~\cite[Prop.\ 4.4]{MRR}). Since there cannot be infinitely many disjoint tangent lines to $\partial D_\rho$, we can proceed likewise with the rest of them to ensure that they do not intersect any other divergence line of the subsequence. By successively enlarging the radius $\rho$, we can continue meeting new divergence lines, in which case we apply the same reasoning. Since there are countably many such steps (each divergence line removes an open subset of $\partial D_r$), we can finally take a diagonal sequence of all the involved subsequences to get rid of all possible intersections.

Note that the divergence lines act as boundary components of a domain of convergence, in the sense that there is an open subset $\Omega\subset D_r$ containing the origin such that $\partial\Omega\cap D_r$ is a union of disjoint divergence lines, and the surfaces $\Sigma_n$ subconverge uniformly on compact subsets of $\Omega$ to a minimal graph over $\Omega$ while they diverge to $\pm\infty$ along the (disjoint) components of $\partial\Omega\cap D_r$.

Given $r'>r$, we can use the above argument to find another minimal graph $\Sigma_{r',\infty}$, but this can be achieved by starting with the subsequence that already converges to $\Sigma_{r,\infty}$ instead of the original sequence $\Sigma_n$. Doing so, the new set of divergence lines contains the divergence lines obtained so far. This means that, if the divergence segments for $\Sigma_r$ intersect outside $D_r$, then we can apply this reasoning for some $r'>r$ such that the intersection occurs in $D_{r'}$, and then pass to a subsequence that eliminates one of the divergence lines. All in all, the subsequence can be chosen such that there are at most two parallel divergence segments. 

This process can be done for an increasing sequence of radii $r_n\to\infty$ to obtain minimal graphs $\Sigma_{r_n,\infty}$, each of them by further refining the subsequence that converges to the previous one, so that $\Sigma_{r_n,\infty}\subset\Sigma_{r_{n+1},\infty}$ for all $n$. A diagonal argument implies that there is a complete minimal graph $\Sigma_\infty$ with angle function $1$ at the origin over all the plane $\mathbb{R}^2$ or a halfplane or a strip, depending on whether there are $0$, $1$ or $2$ divergence lines, respectively. If $H$ and $\nabla H$ are bounded, the last two cases can be discarded by Lemma~\ref{lem:halfspace}, so the dual entire spacelike graph in $\mathbb{L}^3$ has prescribed mean curvature $H$.
\end{proof}

\begin{remark}\label{rmk:prescribed-normal}
The same argument shows that there is always an entire minimal graph in $\mathbb{R}^3_H$ with prescribed unit normal at some fixed $p\in\mathbb{R}^2$. 

To see this, note that the (upward-pointing) unit normal of the minimal graph in $\mathbb{R}^3_H$ is given by $N=-\frac{\alpha}{\omega}E_1-\frac{\beta}{\omega}E_2+\frac{1}{\omega}E_3$, so that we can prescribe $\alpha$ and $\beta$ at $p$ by choosing all the elements of the convergent sequence with this normal at $p$ (the map $\eta$ in Lemma~\ref{lem:angle-1} is a bijection). By the twin relations~\eqref{eqn:twin}, this means that we can prescribe the (timelike) unit normal of the dual graph in $\mathbb{L}^3$ since it is given by $\widetilde N=\frac{\widetilde\alpha}{\widetilde\omega}\widetilde E_1+\frac{\widetilde\beta}{\widetilde\omega}\widetilde E_2+\frac{1}{\widetilde\omega}\widetilde E_3=-\beta\widetilde{E}_1+\alpha\widetilde{E}_2+\omega\widetilde{E}_3$.
\end{remark}

\begin{remark}
This technique works in any Killing submersion $\mathbb{E}(M,\tau,\mu)$ provided that there is an exhaustion of $M$ by disks whose boundaries are convex in the conformal metric $\mu^2\df s_M^2$. For instance, it works in a unit Killing submersion over a Hadamard surface or in $\mathrm{Sol}_3=\mathbb{E}(\mathbb{H}^2,0,x^2)$ (this is the metric given by Nguyen~\cite{Nguyen} in her solution to the Jenkins--Serrin problem). 

The point is that, if not an entire graph, the domain of the constructed complete graph is bounded by disjoint geodesics (in the conformal metric $\mu^2\df s_M^2$, see~\cite[Lem.~3.6]{DMN}). The hypothesis of Theorem~\ref{thm:L3} on $H$ are just the hypothesis of Lemma~\ref{lem:halfspace}, so an improved halfspace theorem with respect to vertical planes in Killing submersions would allow us to show the existence of solutions to the prescribed mean curvature equation under milder hypothesis. As a matter of fact, we believe that all hypothesis on $H$ and $\nabla H$ can be dropped, as it happens in the rotational and translational cases, which are analyzed next in the more general setting of warped products (the space $\mathbb{L}^3$ is recovered for $\lambda\equiv\mu\equiv 1$).
\end{remark}

\begin{proposition}\label{prop:rotations-translations}
The Lorentzian warped product $\mathbb{L}(M,0,\mu)$, where we consider the base surface $M=(\Omega,\lambda^2(\df x^2+\df y^2))$, admits an entire spacelike graph with prescribed mean curvature $H$ under any of the following two assumptions:
\begin{enumerate}[label=\emph{(\alph*)}]
  \item $\Omega\subseteq\mathbb{R}^2$ is a disk centered at the origin with radius $0<R\leq+\infty$ and $\lambda,\mu,H\in\mathcal{C}^\infty(\Omega)$ are radial functions (such that $\lambda,\mu>0$). 
  \item $\Omega\subseteq\mathbb{R}^2$ is a strip of width $0<R\leq+\infty$ and $\lambda,\mu,H\in\mathcal{C}^\infty(\Omega)$ are functions invariant by translations along the strip (such that $\lambda,\mu>0$). 
\end{enumerate}
\end{proposition}

\begin{proof}
In the rotational case, consider the Riemannian space $\mathbb{E}(M,H,\mu^{-1})$ modeled as $\Omega\times\mathbb{R}$ with the metric $\lambda^2(\df x^2+\df y^2)+\mu^{-2}(\df z+y\,\mathbf{C}\,\df x-x\,\mathbf{C}\,\df y)^2$, where $\mathbf{C}$ is the Calabi potential (see Remark~\ref{rmk:calabi-potential}). It is easy to check that the graph $z=0$ is minimal in this model using Equation~\eqref{eqn:H} and the fact that $\lambda$, $H$ and $\mu$ (and hence $\mathbf{C})$ are radial functions. Therefore, the dual graph in $\mathbb L(M,0,\mu)$ is an entire spacelike graph with mean curvature function $H$.

In item (b) we will consider $\mathbb{E}(M,H,\mu^{-1})$ modeled as $\Omega\times\mathbb{R}$ with the metric 
\[\lambda(x)^2(\df x^2+\df y^2)+\frac{1}{\mu(x)^2}(\df z-f(x)\,\df y)^2,\qquad f(x)=2\int\frac{H(x)\lambda(x)^2}{\mu(x)}\df x.\]
This model is obtained by assuming that the strip runs in the direction of the $x$-axis and integrating~\eqref{eqn:tau-model} with $a\equiv 0$. Again, the graph $z=0$ is minimal by Equation~\eqref{eqn:H} and satisfies $\alpha\equiv 0$ and $\beta$ depends only on the variable $x$. As in item (a), the dual graph in $\mathbb L(M,0,\mu)$ is the desired entire spacelike graph.
\end{proof}

\section{Existence and non-existence of entire graphs}\label{sec:non}

 Given a non-compact Riemannian surface $M$ and a positive function $\mu\in\mathcal{C}^\infty(M)$, we define the \emph{Cheeger constant} of $M$ with density $\mu$ as
\begin{equation}\label{eqn:cheeger}
\mathrm{Ch}(M,\mu)=\inf\left\{\frac{\int_{\partial D}\mu}{\int_D\mu}:D\subset M\text{ open and regular}\right\} 
\geq 0.
\end{equation}
Here, an open subset $D\subset M$ is said \textit{regular} if it is relatively compact and its boundary is piecewise smooth so the quotient in~\eqref{eqn:cheeger} makes sense. Note that $\mathrm{Ch}(M,\mu)$ remains invariant when changing $\mu$ into $a\mu$ for any positive constant $a$.

\begin{theorem}\label{thm:NON}
Let $M$ be a non-compact simply-connected surface and consider an arbitrary positive function $\mu\in\mathcal{C}^\infty(M)$. 
\begin{enumerate}[label=\emph{(\alph*)}]
 \item Given $H\in \mathcal{C}^\infty(M)$ such that $\inf_M|H|>\frac{1}{2}\mathrm{Ch}(M,\mu)$, the space $\mathbb{E}(M,\tau,\mu)$ admits no entire graphs with mean curvature $H$ for any $\tau\in\mathcal{C}^\infty(M)$.
 \item Given $\tau\in\mathcal{C}^\infty(M)$ such that $\inf_M|\tau|>\frac{1}{2}\mathrm{Ch}(M,\mu)$, the space $\mathbb{L}(M,\tau,\mu^{-1})$ admits neither complete spacelike surfaces nor entire spacelike graphs.
\end{enumerate}
\end{theorem}

\begin{proof}
We will use a standard argument due to Heinz~\cite{He55} to get item (a). Let us argue by contradiction supposing that such an entire graph exists and it is given by $u\in\mathcal{C}^\infty(M)$ with respect to some initial section. Applying the divergence theorem to the mean curvature equation given by Proposition~\ref{prop:mean-curvature} over an open regular domain $D\subset M$ and Cauchy-Schwarz inequality, we get
\begin{align}\label{thm:NON:eqn1}
  2H_0\int_D\mu 
 \leq \int_{D}\!\!\diver\!\left(\!\!\frac{\mu^2\, Gu}{\sqrt{1+\mu^2\|Gu\|^2}}\!\right)  
 = \int_{\partial D}\mu\left\langle\!\frac{\mu\, Gu}{\sqrt{1+\mu^2\|Gu\|^2}},\eta\!\right\rangle< \int_{\partial D}\!\!\mu,
\end{align}
where $\eta$ is an outer unit conormal to $D$ along its boundary and $H_0=\inf_M(H)$. The condition $\inf|H|>\frac{1}{2}\mathrm{Ch}(M,\mu)\geq 0$ implies that $H$ has a sign. If $H_0>0$ (and hence $H>0$), since~\eqref{thm:NON:eqn1} holds for all regular domains $D$, we find that
\[\textstyle H_0=\inf_M(H)=\inf_M|H|<\tfrac{1}{2}\mathrm{Ch}(M,\mu),\]
contradicting the hypothesis in the statement. Otherwise,
we have $H_0<0$, so we change the sign of the normal in the above argument to get that $-2H_0\int_D\mu\leq\int_{\partial D}\mu$, so $-H_0=\inf_M|H|<\frac{1}{2}\mathrm{Ch}(M,\mu)$ and we get a contradiction again.

As for item (b), we will reason by contradiction again: if there is a complete spacelike surface $\widetilde\Sigma\subset\mathbb{L}(M,\tau,\mu^{-1})$, then $\widetilde\Sigma$ would be an entire graph (the proof is the same as in~\cite[Lem.~4.11]{LeeMan} since the projection $\pi|_{\widetilde\Sigma}:\widetilde\Sigma\to M$ is distance non-decreasing) so its dual surface $\Sigma\subset\mathbb E^3(M,H,\mu)$ is an entire graph, where $H$ denotes the mean curvature of $\widetilde\Sigma$. Now, $\tau$ becomes the mean curvature of $\Sigma$ and verifies $\inf_M|\tau|>\frac{1}{2}\mathrm{Ch}(M,\mu)$, in contradiction with item (a).
\end{proof}

In $\mathbb{E}(\kappa,\tau)$-spaces, the Cheeger constant (with density $\mu\equiv 1$) is given by 
\[\mathrm{Ch}(\mathbb{M}^2(\kappa),1)=\begin{cases}
\sqrt{-\kappa}&\text{if }\kappa\leq 0,\\0&\text{if }\kappa\geq 0.\end{cases}\]
Consequently, the value $\frac{1}{2}\mathrm{Ch}(M,\mu)$ given by Theorem~\ref{thm:NON} is nothing but the \emph{critical} mean curvature in $\mathbb{E}(\kappa,\tau)$-spaces. It is well known that this also reflects the dichotomy between the existence of entire $H$-graphs and the existence of compact $H$-surfaces (both types of surfaces cannot coexist by the maximum principle, with the exception of horizontal slices in $\mathbb{S}^2(\kappa)\times\mathbb{R}$). 

We will show next that this dichotomy extends to rotationally invariant Riemannian warped products by a means of a tricky application of the duality. However, in this general case, we will find another type of surface that we will call $H$-\emph{cigar} since it is a graph over a disk with asymptotic value $+\infty$ on the boundary of the disk, see Figure~\ref{fig:cigar}. It can be thought of as a half-sphere of infinite height.

\begin{figure}
\includegraphics[width=0.6\textwidth]{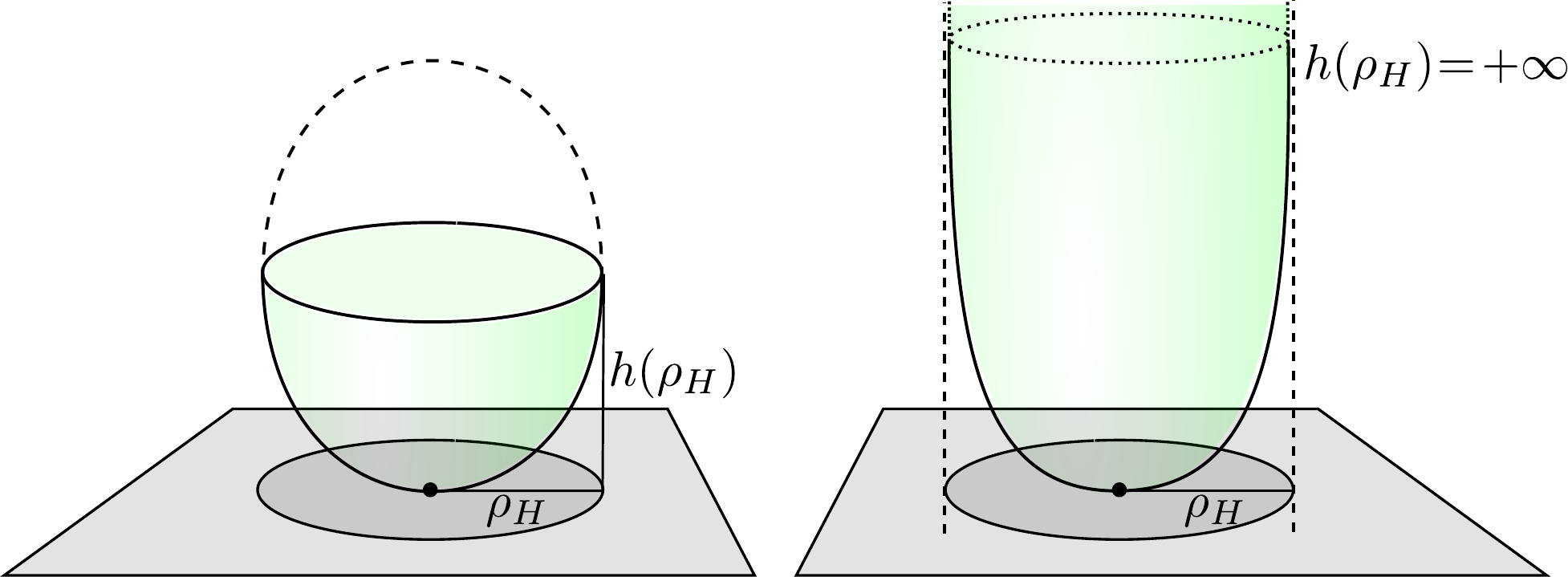}
\caption{An $H$-halfsphere (left) and an $H$-cigar (right).}\label{fig:cigar}
\end{figure}

\begin{theorem}\label{thm:rotational}
Let $\Omega\subseteq\mathbb{R}^2$ be a disk centered at the origin with radius $R\in (0,+\infty]$  and consider the rotationally invariant Riemannian surface $M=(\Omega,\lambda^2(dx^2+dy^2))$ with $\lambda,\mu\in\mathcal{C}^\infty(\Omega)$ radial functions such that $\lambda,\mu>0$. Given a constant $H\geq 0$:
\begin{enumerate}[label=\emph{(\alph*)}]
 \item If $H>\frac{1}{2}\mathrm{Ch}(M,\mu)$, then $\mathbb{E}(M,0,\mu)$ admits an embedded rotationally invariant $H$-sphere or $H$-cigar.
 \item If $H\leq\frac{1}{2}\mathrm{Ch}(M,\mu)$, then $\mathbb{E}(M,0,\mu)$ admits a rotationally invariant entire $H$-graph. 
\end{enumerate}
As a consequence, $\mathbb{E}(M,0,\mu)$ does not admit compact $H$-surfaces for $H\leq\frac{1}{2}\mathrm{Ch}(M,\mu)$, and does not admit entire $H$-graphs with $H>\frac{1}{2}\mathrm{Ch}(M,\mu)$ either.
\end{theorem}

\begin{proof}
In the sequel we will write $\mu=\mu(\rho)$ and  $\lambda=\lambda(\rho)$, where $\rho=(x^2+y^2)^{1/2}$. Consider the Lorentzian space $\mathbb L^3(M,H,\mu^{-1})$, whose Calabi potential with respect to the conformal parametrization is also radial, given by $\mathbf{C}(\rho)=2H c(\rho)$, where 
\[c:[0,R)\to\mathbb{R},\qquad c(\rho)=\int_0^1s\,\lambda(s\rho)^2\mu(s\rho)\,ds=\frac{1}{\rho^2}\int_0^\rho s\lambda(s)\mu(s)^2\df s\geq 0.\]
This means that $\mathbb L^3(M,H,\mu^{-1})$ is modeled as $\Omega\times\mathbb{R}$ with metric
\[\lambda^2(\df x^2+\df y^2)-\mu^{-2}\left(\df z-y\,\mathbf{C}\,\df x+x\,\mathbf{C}\,\df y\right)^2,\]
which follows from Equation~\eqref{eqn:ambient-metric} and Remark~\ref{rmk:calabi-potential} with $\widetilde a=\frac{2Hyc}{\lambda}$ and $\widetilde b=\frac{-2Hxc}{\lambda}$. Equation~\eqref{eqn:H} easily implies that the graph $z=0$ is maximal in $\mathbb L^3(M,H,\mu^{-1})$ and the spacelike condition~\eqref{eqn:causality} for this graph reads
\begin{equation}\label{thm:rotational:eqn1}
\mu(\rho)-\frac{2H\rho\, c(\rho)}{\lambda(\rho)}>0.
\end{equation}
Since~\eqref{thm:rotational:eqn1} holds true for $\rho=0$, it must still hold true in a neighborhood of $0$, so there is some maximal radius $\rho_H\in(0, R]$ such that~\eqref{thm:rotational:eqn1} is satisfied for $0\leq\rho<\rho_H$. Theorem~\ref{thm:duality} gives a dual $H$-graph in $\mathbb{E}(M,0,\mu)$ over the disk of radius $\rho_H$. As $\mathbb{E}(M,0,\mu)$ has zero bundle curvature, we will choose $a=b=0$ in~\eqref{eqn:ambient-metric} and model it as $\Omega\times\mathbb{R}$ with the metric $\lambda^2(\df x^2+\df y^2)+\mu^2\df z^2$. In this model we parametrize the aforesaid dual $H$-graph as $z=u(x,y)$ for some smooth function $u$ on the of radius $\rho_H$. The twin relations~\eqref{eqn:twin-u} give the derivatives of $u$:
\begin{equation}\label{thm:rotational:eqn2}
u_x=\frac{2Hxc}{\mu\sqrt{\mu^2-\frac{4H^2c^2}{\lambda^2}(x^2+y^2)}},\qquad u_y=\frac{2Hyc}{\mu\sqrt{\mu^2-\frac{4H^2c^2}{\lambda^2}(x^2+y^2)}},\end{equation}
whence $yu_x-xy_y=0$ and $u$ also defines a rotationally invariant surface in $\mathbb{E}(M,0,\mu)$. In particular, we can reparametrize the graph of $u$ as 
\[(\rho,\theta)\mapsto(\rho\sin(\theta),\rho\cos(\theta),h(\rho))\in\Omega\times\mathbb{R}\equiv\mathbb{E}(M,0,\mu),\]
where $0\leq\rho<\rho_H$ and $\theta\in\mathbb{R}$. The profile function $h$ is given by
\[h(\rho)=\int_0^\rho\frac{2H r\,c(r)\,\df r}{\mu(r)\sqrt{\mu(r)^2-\frac{4H^2r^2c(r)^2}{\lambda(r)^2}}}=\int_0^{\rho}\frac{g_1(r)\,\df r}{\sqrt{1-g_2(r)^2}},\]
where $g_1(r)=\frac{2Hrc(r)}{\mu(r)^2}$ and $g_2(r)=\frac{2Hrc(r)}{\lambda(r)\mu(r)}$ are non-negative functions defined for all $r\in[0,R)$ which only vanish at $r=0$. We will distinguish three cases:

\noindent\textbf{Case 1.} If $\rho_H=R$, then $z=0$ is an entire maximal graph so the dual surface $z=u(x,y)$ is an entire rotationally invariant $H$-graph.

\noindent\textbf{Case 2.} Assume that $\rho_H<R$ and $g_2'(\rho_H)\neq 0$. Since $\rho\mapsto h(\rho)$ is increasing and the function $\varphi=g_2'g_1^{-1}$ is continuous and bounded away from zero in a neighborhood of $\rho_H$, it follows that $h(\rho_H)<+\infty$ if and only if
\[\int_0^{\rho_H}\frac{\varphi(r)g_1(r)\,\df r}{\sqrt{1-g_2(r)^2}}=\int_0^{\rho_H}\frac{g_2'(r)\,\df r}{\sqrt{1-g_2(r)^2}}<+\infty.\]
The last integral equals $\arcsin(g_2(\rho_H))<+\infty$, so this argument shows that $h(\rho_H)<+\infty$ and the boundary of the graph $z=u(x,y)$ lies in the slice $\Omega\times\{h(\rho_H)\}$. The graph meets the slice orthogonally since $g_2(\rho_H)=1$ by the maximality of $\rho_H$, whence the angle function of $z=u(x,y)$, given by $\nu(\rho)=\mu(\rho)\sqrt{1-g_2(\rho)^2}$ tends to zero as $\rho\to\rho_H$. Moreover, the transformation $(x,y,z)\mapsto (x,y,2h(\rho_H)-z)$ is an isometry in $\mathbb{E}(M,0,\mu)$ keeping the (totally geodesic) slice $\Omega\times\{h(\rho_H)\}$ fixed, so the graph can be reflected about this slice to get an embedded $H$-sphere.

\noindent\textbf{Case 3.} Finally, assume that $\rho_H<R$ and $g_2'(\rho_H)=0$. Let $\Gamma$ be the vertical cylinder of equation $x^2+y^2=\rho_H^2$, which has constant mean curvature 
\begin{equation}\label{thm:rotational:eqn3}
\begin{aligned}
2H_\Gamma&\stackrel{(1)}{=}\frac{(\lambda\mu)'(\rho_H)}{\lambda(\rho_H)^2\mu(\rho_H)}+\frac{1}{\rho_H\lambda(\rho_H)}\stackrel{(2)}{=}\frac{1}{\rho_H\lambda(\rho_H)}+\frac{c(\rho_H)+\rho_Hc'(\rho_H)}{\rho_H\lambda(\rho_H)c(\rho_H)}\\
&\stackrel{(3)}{=}\frac{1}{\rho_H\lambda(\rho_H)}+\frac{\mu(\rho_H)\lambda^2(\rho_H)-c(\rho_H)}{\rho_H\lambda(\rho_H)c(\rho_H)}\\
&\stackrel{(4)}{=}\frac{1}{\rho_H\lambda(\rho_H)}+\frac{\mu(\rho_H)\lambda^2(\rho_H)-\frac{\lambda(\rho_H)\mu(\rho_H)}{2H\rho_H}}{\rho_H\lambda(\rho_H)\frac{\lambda(\rho_H)\mu(\rho_H)}{2H\rho_H}}=2H.
\end{aligned}\end{equation}
The equality (1) in~\eqref{thm:rotational:eqn3} to compute the mean curvature of a vertical cylinder follows from~\cite[Eq.~2.10]{DMN}; (2) uses the condition $g_2'(\rho_H)=0$, in which we solve for $(\lambda\mu)'(\rho_0)$; (3) uses the identity $\frac{\df}{\df r}(rc(r))=\mu(r)\lambda(r)^2$, which in turn follows from~\eqref{eqn:tau-model} and the fact that the bundle curvature of $\mathbb{L}(M,H,\mu^{-1})$ is $H$ (note that the Killing length is $\mu^{-1}$); finally, (4) is a consequence of the fact that $g_2(\rho_H)=1$ by the maximality of $\rho_H$. We will conclude that $h(\rho_H)=+\infty$ by contradiction. If $h(\rho_H)<+\infty$, then the $H$-graph $z=u(x,y)$ lies in the interior of the $H$-cylinder $x^2+y^2=\rho_H^2$. They are tangent along the boundary because $\nu(\rho_H)=\mu(\rho_H)\sqrt{1-g_2(\rho_H)^2}=0$ as in the above item. The boundary maximum principle for $H$-surfaces yields the desired contradiction.

Now observe that~\eqref{thm:rotational:eqn1} implies that $H\mapsto\rho_{H}$ is a continuous and decreasing function of $H$. That means that there exists $H_0\geq 0$ such that $z=u(x,y)$ defines an entire graph for $H\leq H_0$ and an $H$-halfsphere or an $H$-cigar for $H>H_0$ (depending on whether $g_2'(\rho_H)$ vanishes or not). Note that in the case $H=0$, then $u\equiv 0$ is an entire minimal graph.  Recall that entire $H$-graphs and $H$-spheres (or $H$-cigars) cannot coexist due to the maximum principle for $H$-surfaces. Hence, it remains to prove that $H_0=\frac{1}{2}\mathrm{Ch}(M,\mu)$ and we will be done. 

On the one hand, Theorem~\ref{thm:NON} yields non-existence of entire $H$-graphs for $H>\frac{1}{2}\mathrm{Ch(M,\mu)}$, so we deduce that $H_0\leq\frac{1}{2}\mathrm{Ch(M,\mu)}$. On the other hand, let $D_\rho$ be the disk of center $0$ and Euclidean radius $0<\rho<R$. By definition of Cheeger constant,
\begin{equation}\label{thm:rotational:eqn4}
\mathrm{Ch}(M,\mu)\leq\frac{\int_{\partial D_\rho}\mu}{\int_{D_\rho}\mu}=\frac{\int_0^{2\pi}\rho\,\lambda(\rho)\mu(\rho)\,\df\theta}{\int_0^{2\pi}\int_{0}^\rho r\,\lambda(r)^2\mu(r)\,\df r\,\df\theta}=\frac{\lambda(\rho)\mu(\rho)}{\rho\, c(\rho)}
\end{equation}
for all $0<\rho<R$, where we have used polar coordinates $(r,\theta)$. Given $0\leq H<\frac{1}{2}\mathrm{Ch}(M,\mu)$, the estimate~\eqref{thm:rotational:eqn4} implies that the causality condition~\eqref{thm:rotational:eqn1} holds for all $0<\rho<R$, so the above construction (Case 1) provides an entire $H$-graph for all $H<\frac{1}{2}\mathrm{Ch}(M,\mu)$. It follows that $H_0=\frac{1}{2}\mathrm{Ch(M,\mu)}$.
\end{proof}

\begin{remark}
The $H$-cigars are tangent to a vertical cylinder at infinity of equation $x^2+y^2=\rho_H^2$. This vertical cylinder (which is homogeneous as a surface of $\mathbb{E}(M,0,\mu$), has the same constant mean curvature $H$ as shown in the proof (Case $3$). One can see the $H$-cigars as solutions to a Jenkins--Serrin problem for $H$-surfaces in $\mathbb{E}(M,0,\mu)$ with just one boundary component.
\end{remark}

\begin{remark}
The proof shows indirectly that the Cheeger constant can be obtained explicitly from the radial geometric data $\lambda$ and $\mu$ as
\[\mathrm{Ch}(M,\mu)=\inf_{0<\rho<R}\frac{\rho\,\lambda(\rho)\mu(\rho)}{\int_0^\rho s\,\lambda(s)^2\mu(s)\,\df s}.\]
The inequality $\leq$ follows directly from the computations for disks $D_\rho$ in the proof of Theorem~\ref{thm:rotational}. Assume by contradiction that a strict inequality $<$ holds. In that case, there exists $H>0$ such that
\[\mathrm{Ch}(M,\mu)<2H<\inf_{0<\rho<R}\frac{\rho\,\lambda(\rho)\mu(\rho)}{\int_0^\rho s\lambda(s)^2\mu(s)\,\df s}.\]
In particular,~\eqref{thm:rotational:eqn1} is satisfied for all $0<\rho<R$, so $z=0$ in $\mathbb{L}(M,H,\mu^{-1})$ defines an entire spacelike maximal graph, and its twin graph in $\mathbb{E}(M,0,\mu)$ has constant mean curvature $H>\frac{1}{2}\mathrm{Ch}(M,\mu)$, in contradiction with item (a) of Theorem~\ref{thm:NON}.
\end{remark}

\end{document}